\UseRawInputEncoding
\documentclass[11pt]{amsart}
\usepackage{amsfonts}
\usepackage{}
\usepackage{mathrsfs}
\usepackage{bbm}
\usepackage{amssymb}
\usepackage{indentfirst}
\usepackage{latexsym,amsfonts,amssymb,amsmath}
\pdfoutput=1
\newtheorem{theorem}{Theorem}[section]
\newtheorem{lemma}[theorem]{Lemma}
\newtheorem{corollary}[theorem]{Corollary}
\newtheorem{question}[theorem]{Question}
\theoremstyle{definition}
\newtheorem{definition}[theorem]{Definition}
\newtheorem{proposition}[theorem]{Proposition}
\theoremstyle{remark}
\newtheorem{remark}[theorem]{Remark}
\newtheorem{example}[theorem]{Example}

\begin{document}

\title
{Subgyrogroups within the product spaces of paratopological gyrogroups}

\author{Ying-Ying Jin}\thanks{}
\address{(Y.-Y. Jin) Department of General Required Courses, Guangzhou Panyu Polytechnic, Guangzhou 511483, P.R. China} \email{yingyjin@163.com, jinyy@gzpyp.edu.cn}
\author{Ye-Qing Sheng}\thanks{}
\address{(Y.Q. Sheng) School of Mathematics and Computational Science, Wuyi University, Jiangmen 529020, China} \email{sheng$_-$yq@126.com}
\author{Yi-Ting Wang}\thanks{}
\address{(Y.-T. Wang) School of Mathematics and Computational Science, Wuyi University, Jiangmen 529020, China} \email{2205616228@qq.com}
\author{Li-Hong Xie*}\thanks{* The corresponding author.}
\address{(L.-H. Xie) School of Mathematics and Computational Science, Wuyi University, Jiangmen 529020, P.R. China} \email{yunli198282@126.com}


\thanks{
This work is supported by the Natural Science Foundation of Guangdong
Province under Grant (Nos. 2021A1515010381; 2020A1515110458). The Innovation Project of Department of Education of Guangdong Province (No. 2022KTSCX145), and the Natural Science Project of Jiangmen City (No. 2021030102570004880), Scientific research project of Guangzhou Panyu Polytechnic  (No. 2022KJ02)}



\begin{abstract} We present a characterization of paratopological gyrogroups that can be topologically embedded as subgyrogroups into a product of first-countable $T_{i}$ paratopological gyrogroups for $i =
0, 1, 2$. Specifically, we demonstrate that
a strongly paratopological gyrogroup
$G$ is topologically isomorphic to a subgyrogroup of a topological product
of first-countable $T_1$ strongly paratopological gyrogroups if and only if $G$ is $T_1$, $\omega$-balanced
and the weakly Hausdorff number of $G$ is countable.
Similarly, we prove that
a strongly paratopological gyrogroup
$G$ is topologically isomorphic to a subgyrogroup of a topological product
of first-countable Hausdorff strongly paratopological gyrogroups if and only if $G$ is Hausdorff, $\omega$-balanced
and the Hausdorff number of $G$ is countable.
\end{abstract}

\maketitle

\section{Introduction}
A gyrogroup, as outlined in Definition \ref{Def:gyr}, is essentially a structure similar to a group, but it notably lacks the associative property.
Research into gyrogroups, initiated in the early 1980s by Ungar \cite{Ung1}, originally formulated the idea of gyrovector spaces and subsequently expanded it to include the broader concept of a gyrogroup.
A {\it paratopological gyrogroup} is a gyrogroup endowed with a topology for which binary operation
in the gyrogroup is jointly continuous. If, additionally, the inversion in a paratopological gyrogroup is
continuous, then it is called a {\it topological gyrogroup} (see Definition \ref{def2.11}).

According to \cite{Tka}, given a topological property $\mathcal{P}$, we say that a paratopological (topological) gyrogroup $G$
is {\it projectively} $\mathcal{P}$ if for every neighborhood $U$ of the identity in $G$, there exists a continuous homomorphism
$p: G\rightarrow H$ onto a paratopological (topological) gyrogroup $H$ with property $\mathcal{P}$ such that $p^{-1}(V)\subseteq U$, for some neighborhood $V$ of the identity in $H$.

Tkachenko provided a detailed internal analysis of the projectively Hausdorff first-countable (second-countable) paratopological groups through the utilization of a cardinal invariant known as the {\it Hausdorff number}, as outlined in \cite[Theorems 2.7, 2.8]{Tka}. Furthermore, the characteristics of regular paratopological groups are extensively discussed in Theorems 3.6 and 3.8 \cite{Tka}:
A regular paratopological group $H$ can be topologically embedded as a subgroup into a product of regular first-countable
(second-countable) paratopological groups if and only if $H$ has countable index of regularity and is $\omega$-balanced (totally $\omega$-narrow).
This means that for every neighbourhood $U$ of the identity $e$ in $G$, one
can find a countable family $\gamma$ of neighbourhoods of $e$ such that for each $x\in G$ there exists $V\in\gamma$ satisfying
$xVx^{-1}\subseteq U$.
Such a family $\gamma$ is usually called subordinated to $U$.

S\'{a}nchez offered a characterization of subgroups within topological products of families of first-countable (second-countable) $T_i$ paratopological groups, where $i=0, 1$, by employing a cardinal invariant referred to as the symmetry number.
At the same time, the authors in \cite{XL} introduced a same concept known as weakly Hausdorff number.
Through \cite[Theorems 2.17, 2.19]{Sa}, they established that being $\omega$-balanced ( totally $\omega$-narrow) is a necessary and sufficient condition for a $T_0$ paratopological group to be projectively first-countable (second-countable) $T_0$. Furthermore, in \cite[Theorems 2.20, 2.22]{Sa}, it was demonstrated that a $T_1$ paratopological group $G$ can be topologically embedded as a subgroup in a product of $T_1$ first-countable (second-countable) paratopological groups if and only if $G$ is $\omega$-balanced (totally $\omega$-narrow) and maintains a countable symmetry number.

Motivated by the techniques used in \cite{Sa}, we characterize subgyrogroups of topological products of families
of first-countable $T_i$ paratopological gyrogroups for $i=0, 1, 2$. We prove in Theorem \ref{the3.26}
that
a strongly paratopological gyrocommutative gyrogroup
$G$ is projectively
first-countable if and only if $G$ is $T_0$ and $\omega$-balanced.
We demonstrate in Theorem \ref{the3.30s} that
a strongly paratopological gyrogroup
$G$ is topologically isomorphic to a subgyrogroup of a topological product
of first-countable $T_1$ strongly paratopological gyrogroups if and only if $G$ is $T_1$, $\omega$-balanced
and the weakly Hausdorff number of $G$ is countable. This means that for every neighborhood $U$ of the identity 0 in $G$, there exists a countable family $\gamma$ of neighborhoods of 0 such that for all $V \in\gamma$, $\bigcap_{V\in\gamma} (\ominus V)\subseteq U$.
We also show in Theorem \ref{the3.30} that
a strongly paratopological gyrogroup
$G$ is topologically isomorphic to a subgyrogroup of a topological product
of first-countable Hausdorff strongly paratopological gyrogroups if and only if $G$ is Hausdorff, $\omega$-balanced and $Hs(G)\leq\omega$.

There are many results on the question of when a paratopological (semitopological) group is in fact a topological group. For example,
according to Ellis' theorem in \cite{El}, every locally compact Hausdorff semitopological group is a topological group.
Romaguera and Sanchis \cite{RS} generalized the famous Numakura¡¯s theorem \cite{Nu} and showed that
every compact Hausdorff topological semigroup with two-sided cancellation is a topological group.
A conclusion drawn from this result in \cite{RS} is that every compact $T_0$ paratopological group is a topological group.

In 2014, S\'{a}nchez in \cite {Sa} and, independently, L. H. Xie and S. Lin in \cite{XL} proved that every 2-pseudocompact paratopological group $G$ is a topological group if and only if
it has a countable symmetry number. This characterization implies Ravsky's result in \cite{Ra1}.
In Theorem \ref{the3.15} we give necessary and sufficient conditions under which a 2-pseudocompact $T_1$ strongly paratopological gyrogroup turns out to be a strongly topological gyrogroup.

\section{Definitions and preliminaries}
This section contains necessary definitions and background of gyrogroups and paratopological (topological) gyrogroups.
Their basic algebraic and topological properties are included as well.
For notation and terminology not explicitly mentioned here, readers are encouraged to refer to \cite{Arha, En89}.
\subsection{The algebra of gyrogroups}

In this section, we give the relevant definitions, summarize elementary properties
of gyrogroups, gyrogroup homomorphisms, normal subgyrogroups,
and quotient gyrogroups.
Much of this section can be found in \cite{ST, Suk3, Ung1, Ung}.

Let $G$ be a nonempty set, and let $\oplus  : G  \times G \rightarrow G $ be a binary operation on $G $. Then the pair $(G, \oplus)$ is
called a {\it groupoid.}  A function $f$ from a groupoid $(G_1, \oplus_1)$ to a groupoid $(G_2, \oplus_2)$ is said to be
a groupoid homomorphism if $f(x_1\oplus_1 x_2)=f(x_1)\oplus_2 f(x_2)$ for any elements $x_1, x_2 \in G_1$.  In addition, a bijective
groupoid homomorphism from a groupoid $(G, \oplus)$ to itself will be called a groupoid automorphism. We will write $Aut (G, \oplus)$ for the set of all automorphisms of a groupoid $(G, \oplus)$.

The notion of a gyrogroup was first identified by Ungar during his research on Einstein's velocity addition \cite{Ung}.

\begin{definition}\cite[Gyrogroups]{Ung}\label{Def:gyr}
 Let $(G, \oplus)$ be a nonempty groupoid. We say that $(G, \oplus)$ or just $G$
(when it is clear from the context) is a gyrogroup if the followings hold:
\begin{enumerate}
\item[($G1$)] There is an identity element $0 \in G$ such that
$$0\oplus x=x=x\oplus 0\text{~~~~~for all~~}x\in G.$$
\item[($G2$)] For each $x \in G $, there exists an {\it inverse element}  $\ominus x \in G$ such that
$$\ominus x\oplus x=0=x\oplus(\ominus x).$$
\item[($G3$)] For any $x, y \in G $, there exists an {\it gyroautomorphism} $\text{gyr}[x, y] \in Aut(G,  \oplus)$ such that
$$x\oplus (y\oplus z)=(x\oplus y)\oplus \text{gyr}[x, y](z)$$ for all $z \in G$;
\item[($G4$)] For any $x, y \in G$, $\text{gyr}[x \oplus y, y]=\text{gyr}[x, y]$.
\end{enumerate}
\end{definition}

\begin{definition}\cite{Ung}\label{defbox}
Let $(G,\oplus)$ be a gyrogroup with gyrogroup operation (or,
addition) $\oplus$. The gyrogroup cooperation (or, coaddition) $\boxplus$ is a second
binary operation in $G$ given by the equation
$a\boxplus b=a\oplus \text{gyr}[a,\ominus b]b$ for all $a, b\in G$.
The groupoid $(G, \boxplus)$ is called a cogyrogroup, and is said to
be the cogyrogroup associated with the gyrogroup $(G, \oplus)$.

Replacing $b$ by $\ominus b$, we have the identity
$a\boxminus b=a\ominus \text{gyr}[a,b]b$ for all $a, b\in G$, where we use the obvious notation, $a\boxminus b = a\boxplus(\ominus b)$.
\end{definition}

Gyrogroups possess similar algebraic characteristics to groups, and numerous theorems from group theory have been adapted to gyrogroups through the utilization of gyroautomorphisms \cite{Ung1, Ung}.
The following Theorem \ref{the1.3} below summarizes some algebraic properties of gyrogroups.

\begin{theorem}\cite{Ung}\label{the1.3}
Let $(G, \oplus)$ be a gyrogroup. Then, for any $a, b, c\in G$ we have
\begin{enumerate}
\item[(1)] $(a\oplus b)\oplus c=a\oplus(b\oplus \text{gyr}[b, a]c);$\hfill{Right Gyroassociative Law}
\item[(2)] $a\oplus (b\oplus c)=(a\oplus b)\oplus \text{gyr}[a, b]c;$\hfill{Left Gyroassociative Law}
\item[(2)] $\text{gyr}[a, b]=\text{gyr}[a, b\oplus a];$\hfill{Right Loop Property}
\item[(2)] $\text{gyr}[a, b]=\text{gyr}[a\oplus b, b];$\hfill{Left Loop Property}
\item[(3)] $(\ominus a)\oplus(a\oplus b)= b$;
\item[(4)] $(a\ominus b)\boxplus b= a$;
\item[(5)] $(a\boxminus b)\oplus b= a$;
\item[(6)] $\text{gyr}[a, b](c)=\ominus(a\oplus b)\oplus (a\oplus (b\oplus c))$;
\item[(7)] $\ominus(a\oplus b)=\text{gyr}[a, b](\ominus b\ominus a)$;\hfill{Gyrosum Inversion}
\item[(8)] $\text{gyr}[a, b](\ominus x)=\ominus \text{gyr}[a, b]x$;
\item[(9)] $\text{gyr}^{-1}[a, b]=\text{gyr}[b, a]$; \hfill{Inversive symmetry}
\item[(10)] $\ominus(a\boxplus b)= (\ominus b)\boxplus(\ominus a)$; \hfill{The Cogyroautomorphic Inverse Theorem}
\item[(11)] $\text{gyr}[\ominus a, \ominus b]=\text{gyr}[a, b]$; \hfill{Even symmetry}
\item[(12)] $\text{gyr}[a, 0]=\text{gyr}[0, b]=\text{id}_G$.
\end{enumerate}
\end{theorem}

Generally, gyrogroups do not adhere to the associative property. However, they comply with the left and right gyroassociative laws, denoted as Proposition \ref{the1.3} (1) and (2), respectively.
Observe that a group is a type of gyrogroup $(G, \oplus)$ where \text{gyr}$[x, y]$ acts as the identity map for all $x, y$ in $G$. This indicates that gyrogroups serve as a natural extension of the concept of groups.
Specifically, gyrogroups that extend the principles of Abelian groups are assigned a distinct designation.
Similar to the classification in group theory, gyrogroups are categorized as either gyrocommutative or non-gyrocommutative.

\begin{definition}\cite[Gyrocommutative Gyrogroups]{Ung}\label{com}
A gyrogroup $(G,\oplus)$ is gyrocommutative if its binary operation obeys the gyrocommutative law
$$a\oplus b=\text{gyr}[a,b](b\oplus a) $$
for all $a, b\in G$.
\end{definition}
In fact, both the Einstein gyrogroup and the M\"{o}bius gyrogroup are examples of gyrocommutative gyrogroups, as detailed in Example \ref{ex13} and \ref{ex16}. Non-commutative groups are also referred to as non-gyrocommutative gyrogroups.
The subsequent Theorem \ref{the2.5com} encapsulates certain algebraic traits of gyrocommutative gyrogroups.

\begin{theorem}\cite{Ung}\label{the2.5com}
 Let $(G, \oplus)$ be a gyrocommutative gyrogroup. Then, for any $a, b, c\in G$ we have
\begin{enumerate}
\item[(1)] $\ominus(a\oplus b)=\ominus a\ominus b;$\hfill{Gyroautomorphic Inverse Property}
\item[(2)] $a\boxplus b=b\boxplus a;$
\item[(3)] $a\boxplus b=a\oplus((\ominus a\oplus b)\oplus a)$.
\end{enumerate}
\end{theorem}

The initial definition of a subgyrogroup was introduced in \cite[Sect. 4]{Fe}, where it was termed a "gyro-subgroup".

\begin{definition}\cite[Subgyrogroups]{Suk3}
Let $G$ be a gyrogroup. A nonempty subset $H$ of $G$
is a subgyrogroup, written $H\leq G$, if $H$ is a gyrogroup under the operation inherited
from $G$ and the restriction of $\text{gyr}[a,b]$ to $H$ becomes an automorphism of $H$ for all
$a, b\in H$.
\end{definition}
Furthermore, a subgyrogroup $H$ of $G$ is said to be an {\it $L$-subgyrogroup} \cite{Suk3}, denoted by $H\leq_L G$,
if $\text{gyr}[a, h](H) =H$ for all $a\in G$ and $h\in H$.

\begin{proposition}\cite[The Subgyrogroup Criterion]{ST}
A nonempty subset $H$ of $G$ is a
subgyrogroup if and only if $a\in H$ implies $\ominus a\in H$ and $a, b\in H$ implies $a\oplus b\in H$.
\end{proposition}

Subsequently, we investigate the properties of gyrogroup homomorphisms, which are functions that respect the gyrogroup operation, along with normal subgyrogroups. Furthermore, we explore quotient gyrogroups, which are the mathematical structures obtained by partitioning a gyrogroup into a set of cosets relative to a normal subgyrogroup. These concepts are developed in parallel with the analogous concepts in group theory.

Let $G$ and $H$ be gyrogroups. A map $\varphi:G\rightarrow H$ is called a {\it~gyrogroup homomorphism}
if $\varphi(a\oplus b)=\varphi(a)\oplus \varphi(b)$ for all $a, b\in G$.

The existing literature has established that within the realm of group theory, a subgroup is considered normal if, and only if, it serves as the kernel of at least one group homomorphism. Utilizing this definition, the author analogously defined a normal subgyrogroup by adopting a similar approach.

\begin{definition}\cite[Normal Subgyrogroups]{ST}\label{NS}
A subgyrogroup $N$ of a gyrogroup $G$ is
normal in $G$, written $N\trianglelefteq G$, if it is the kernel of a gyrogroup homomorphism of $G$.
\end{definition}

Indeed, an operation is specified on the quotient space $G/\ker\varphi$ in a manner that is inherently straightforward:
$$(a\oplus \ker\varphi)\oplus(b\oplus \ker\varphi)=(a\oplus b)\oplus \ker\varphi, \text~for~all~a, b\in G.$$
It is a well-defined operation. In fact, the
coset space $G/\ker\varphi$ forms a gyrogroup, called a {\it~quotient gyrogroup}.
The map $\pi: G\rightarrow G/\ker\varphi$ given by $\pi(a)=a\oplus \ker\varphi$ defines a surjective
gyrogroup homomorphism, which will be referred to as the {\it~canonical projection} \cite{ST}.

\begin{theorem}\label{pro2.9s}
Let $G$ be a gyrocommutative gyrogroup and $N\unlhd G$, then the quotient gyrogroup $G/N$ is also a gyrocommutative gyrogroup.
\end{theorem}
\begin{proof}
Let $\pi: G\rightarrow G/N$ be the canonical projection defined by $\pi(a)=a\oplus N$.
Put $a, b\in G$, and consider their respective cosets in the quotient gyrogroup \( G/N \), denoted by \( a \oplus N \) and \( b \oplus N \).
Then we have
\begin{align*}
&(a\oplus N)\oplus(b\oplus N)=\pi(a)\oplus \pi(b)
\\&=\pi(a\oplus b)
\\&=\pi(\text{gyr}[a,b](b\oplus a))\quad  \text{by Definition \ref{com}}
\\&=\text{gyr}[\pi(a),\pi(b)](\pi(b)\oplus \pi(a))  \quad  \text {since}~\pi~\text {is a surjective homomorphism}
\\&=\text{gyr}[a\oplus N,b\oplus N]((b\oplus N)\oplus (a\oplus N)).
\end{align*}
Hence, $G/H$ is a gyrocommutative gyrogroup.
\end{proof}

\begin{proposition}\cite{Suk3}\label{pro2.2}
 Let $G$ be a gyrogroup and let $X\subseteq G$. Then the following are equivalent:
 \begin{enumerate}
\item[($1$)] $\text{gyr}[a,b](X)\subseteq X$ for all $a, b\in G$;
\item[($2$)] $\text{gyr}[a,b](X) = X$ for all $a, b\in G$.
\end{enumerate}
\end{proposition}

The following theorem demonstrates a characteristic property of a normal subgyrogroup.

\begin{theorem}\cite{ST}\label{the2.5}
Let $N$ be a subgyrogroup of a gyrogroup $G$. Then $N\unlhd G$ if and only if
$a\oplus(N\oplus b)=(a\oplus b)\oplus N=(a\oplus N)\oplus b$
for all $a, b\in G$.
\end{theorem}


The theorem below provides a necessary and sufficient condition for normality of a
subgyrogroup. In this paper, $\text{gyr}[a,b](V)$ denotes $\{\text{gyr}[a,b](v): v\in V\}$.

\begin{theorem}\label{pro2.11s}
Let $N$ be a subgyrogroup of a gyrogroup $G$. Then $N\unlhd G$ if and only if the following conditions hold:
\begin{enumerate}
\item[($1$)] $\text{gyr}[a,b](N)\subseteq N$ for all $a, b\in G$;
\item[($2$)] $\ominus(a\oplus b)\oplus ((N\oplus a)\oplus b)=N$ for all $a, b\in G$.
\end{enumerate}
\end{theorem}

\begin{proof}
To establish the necessity, let us assume $N \unlhd G$. Then, Theorem \ref{the2.5} applies.
According to Theorem \ref{the2.5}, when $a = 0$, it follows that $N \oplus b = b \oplus N $ for all $b \in G $. For all $a, b\in G$, we have
\begin{align*}
&a\oplus (N\oplus b)=(a\oplus b)\oplus N
\\&\Leftrightarrow \ominus(a\oplus b)\oplus(a\oplus (N\oplus b))=N \quad \text{Left cancellation law~}
\\&\Leftrightarrow \ominus(a\oplus b)\oplus(a\oplus (b\oplus N))=N \quad \text{since $N \oplus b = b \oplus N $~}
\\&\Leftrightarrow \text{gyr}[a,b](N)=\ominus(a\oplus b)\oplus(a\oplus (b\oplus N))=N \quad \text{by Theorem \ref{the1.3}(6)~}.
\end{align*}
So the condition (1) holds.

For all $a, b\in G$, we have
\begin{align*}
&(a\oplus N)\oplus b=(a\oplus b)\oplus N
\\&\Leftrightarrow \ominus(a\oplus b)\oplus((a\oplus N)\oplus b)=N \quad \text{Left cancellation law~}
\\&\Leftrightarrow \ominus(a\oplus b)\oplus((N\oplus a)\oplus b)=N \quad \text{since $N \oplus a=a\oplus N $~}.
\end{align*}
So the condition (2) holds.

To prove the sufficiency, take $a, b\in G$. According to (2), when $b= 0$, it follows that $\ominus a\oplus(N \oplus a)=N$ for all $a\in G $, which implies that $N \oplus a=a\oplus N $. Then we have
\begin{align*}
&\text{gyr}[a,b](N)\subseteq N
\\&\Leftrightarrow \text{gyr}[a,b](N)=N \quad  \text{by Proposition \ref{pro2.2}}
\\&\Leftrightarrow \ominus(a\oplus b)\oplus(a\oplus (b\oplus N))=N \quad\quad\quad \text{by Theorem \ref{the1.3}(6)~}
\\&\Leftrightarrow \ominus(a\oplus b)\oplus(a\oplus (N\oplus b))=N \quad \text{since $N \oplus b = b \oplus N $~}
\\&\Leftrightarrow a\oplus (N\oplus b)=(a\oplus b)\oplus N \quad \text{Left cancellation law}.
\end{align*}
We can get
\begin{align*}
&\ominus(a\oplus b)\oplus((N\oplus a)\oplus b)=N
\\&\Leftrightarrow \ominus(a\oplus b)\oplus((a\oplus N)\oplus b)=N \quad \text{since $N \oplus a=a\oplus N $~}
\\&\Leftrightarrow (a\oplus N)\oplus b=(a\oplus b)\oplus N \quad \text{Left cancellation law}.
\end{align*}
Therefore, based on Theorem \ref{the2.5}, it is concluded that $N\unlhd G$.
\end{proof}

Since in Topology 'normal' refers to a separation property of spaces, we will use the term
'invariant' to denote this property of subgyrogroups.

\subsection{The paratopological gyrogroup}
The family of open neighborhoods of the
identity 0 in $G$ will be denoted by $\mathcal{U}$.
No separation restrictions on the topological spaces considered in this paper are imposed unless we mention them explicitly.
Below $\psi(X)$, $\chi(X)$ and $l(X)$ denote the pseudocharacter, character and Lindel\"{o}f number of aspace $X$, respectively.
Moreover, the set of all positive integers denoted by $\mathbb{N}$ and the first infinite ordinal denoted by $\omega$.
We are prepared to present the definition of gyrogroups as follows.

\begin{definition}\cite{Atip}\label{def2.11}
A triple $(G, \tau,  \oplus)$ is called a {\it topological gyrogroup} if and only if
\begin{enumerate}
\item[(1)] $(G, \tau)$ is a topological space;
\item[(2)] $(G, \oplus)$ is a gyrogroup;
\item[(3)] The binary operation $\oplus:G \times G\rightarrow G$ is continuous where $G\times G$ is endowed with the product topology
and the operation of taking the inverse $\ominus(\cdot ) : G  \rightarrow G $, i.e. $x\rightarrow\ominus x$, is continuous.
\end{enumerate}
\end{definition}

If a triple $( G, \tau,  \oplus)$ satisfies the first two conditions and its binary operation is continuous, we call such
triple a {\it paratopological gyrogroup} \cite{Atip1}. Sometimes we will just say that $G$ is a topological gyrogroup (paratopological gyrogroup) if the binary operation and the topology are clear from the context.

\begin{example}\cite[Example 3]{Atip}\label{ex13}
Let's examine the set of all Einsteinian velocities, characterized as follows:
$$\mathbb{R}^3_{c} = \{\mathbf{v} \in \mathbb{R}^3 : \|\mathbf{v}\| < c\}.$$
In this context, $c$ represents the speed of light in a vacuum, and $\|\mathbf{v}\|$ denotes the Euclidean norm of a vector $\mathbf{v}$ in $\mathbb{R}^3$. This set, being a subset of a topological space (namely, $\mathbb{R}^3$ with its standard topology), naturally forms a topological space itself. We then introduce the operation of Einstein addition, $\oplus_E: \mathbb{R}^3_{c} \times \mathbb{R}^3_{c} \to \mathbb{R}^3_{c}$, for any vectors $\mathbf{u}, \mathbf{v}$ in $\mathbb{R}^3_{c}$, as
$$\mathbf{u} \oplus_E \mathbf{v} = \frac{1}{1 + \frac{\mathbf{u} \cdot \mathbf{v}}{c^2}} \left( \mathbf{u} \oplus \frac{1}{\gamma_{\mathbf{u}}} \mathbf{v} + \frac{1}{c^2} \frac{\gamma_{\mathbf{u}}}{1 + \gamma_{\mathbf{v}}} (\mathbf{u} \cdot \mathbf{v}) \mathbf{u} \right),$$
where $\mathbf{u} \cdot \mathbf{v}$ is the standard dot product of vectors in $\mathbb{R}^3$, and the gamma factor $\gamma_{\mathbf{u}}$ within the $c$-ball is defined by
$$\gamma_{\mathbf{u}} = \frac{1}{\sqrt{1 - \frac{\mathbf{u} \cdot \mathbf{u}}{c^2}}}.$$
As detailed in section 3.8 (page 91) of \cite{Ung}, the pair $(\mathbb{R}^3_{c}, \oplus_E)$ forms a gyrocommutative gyrogroup, called an
{\it Einstein gyrogroup}, but not a group.
Moreover, given the standard topology derived from $\mathbb{R}^3$, it is evident that the operation $\oplus_E$ is continuous.
Additionally, it has been established that the inverse of any $\mathbf{u}$ in $\mathbb{R}^3_c$ is $-\mathbf{u}$, thereby ensuring that the inversion operation is continuous as well. Thus, $\mathbb{R}^3_{c}$ constitutes a topological gyrogroup, distinguishing itself from being a topological group, a semitopological group, or a paratopological group, in a rephrased manner.
\end{example}

\begin{example}
Suppose that $(\mathbb{R}, \tau_s)$ is Sorgenfrey line and $(G, \tau)$ is a topological gyrogroup.
Then $\mathbb{R}\times G$ with
product topology is a paratopological gyrogroup and not a topological gyrogroup.
\end{example}

Next, we present the definition of strongly paratopological gyrogroups, which is crucial for the content of this paper.

\begin{definition}\cite{BL}\label{defst}
Let $( G, \tau,  \oplus)$ be a topological gyrogroup. We say that $G$ is a strongly topological gyrogroup if
there exists a neighborhood base $\mathcal{U}$ of the identity 0 in $G$ such that, for every $U\in \mathcal{U}$,
$\text{gyr}[x, y](U)=U$ for any $x, y\in G$.
\end{definition}

A paratopological gyrogroup $( G, \tau,  \oplus)$ is called a {\it strongly paratopological gyrogroup} \cite{JX2} if
there exists a neighborhood base $\mathcal{U}$ of the identity 0 in $G$ such that, for every $U\in \mathcal{U}$,
$\text{gyr}[x, y](U)=U$ for any $x, y\in G$.

\begin{example}\cite[Example 2]{Atip}\label{ex16}
The M\"{o}bius gyrogroup $\mathbb{D}$ is the complex open unit disk $\{z \in \mathbb{C}:|z| < 1\}$ equipped with a M\"{o}bius addition $\oplus_{M}: \mathbb{D} \times \mathbb{D} \to \mathbb{D}$ defined by
$$a\oplus_{M} b=\frac{a+b}{1+\bar{a}b}, \text{where~} a, b \in \mathbb{D}.$$
 It also features a gyroautomorphism $\text{gyr}[a, b]$ defined by
$$\text{gyr}[a,b](c)=\frac{1+a\bar{b}}{1+\bar{a}b}c, \text{where~} a, b, c \in \mathbb{D}.$$
Indeed, the M\"{o}bius gyrogroup, denoted as $(\mathbb{D}, \oplus_M)$, has been established as a gyrocommutative gyrogroup, as referenced in \cite{Ung}.
When the M\"{o}bius gyrogroup $(\mathbb{D}, \oplus_M)$ is equipped with the standard topology, it becomes a topological gyrogroup, although it does not qualify as a topological group.
For each $n \in \mathbb{N}$, we define the set $U_n=\{x \in \mathbb{D} : |x| \leq \frac{1}{n}\}$.
The family $\mathcal{U} = \{U_n : n \in \mathbb{N}\}$ serves as a neighborhood base at the origin 0. It is also noted that $|\frac{1+a\overline{b}}{1+\overline{a}b}|=1$. Hence, for all $x, y \in \mathbb{D}$ and for any $U \in \mathcal{U}$, it is determined that $\text{gyr}[x, y](U)\subseteq U$. From this inclusion, and by reference to Proposition \ref{pro2.2}, we deduce that $\text{gyr}[x, y](U)=U$. This confirms that $(\mathbb{D}, \oplus)$ is a strongly topological gyrocommutative gyrogroup.
\end{example}

We list the following known conclusions, which will be used in the proof of our paper.

\begin{proposition}\cite{Atip1}\label{pro23}
Let $G$ be a paratopological gyrogroup, $x, y\in G$ and $A, B\subseteq G$.
\begin{enumerate}
\item[(1)] The left translation $L_{x}: G \rightarrow G$, where $L_{x}(y) = x \oplus y$ for every $y\in G$, is
homeomorphism;
\item[(2)] $A$ is closed if and only if $x\oplus A$ is closed;
\item[(3)] $A$ is open if and only if $x\oplus A$ and $B\oplus A$ are open;
\end{enumerate}
\end{proposition}

In Section 5 of \cite{Cai}, it is established that the right translation operation in a paratopological loop results in a homeomorphism. Given that every gyrogroup is recognized as a left Bol loop \cite{Sa}, it can be inferred that a paratopological gyrogroup is also a paratopological loop. Consequently, it is deduced that the right translation within a paratopological gyrogroup is indeed a homeomorphism. Next, we will provide a detailed proof of this conclusion.

\begin{proposition} \label{the3.22s}
Let $(G, \oplus)$ be a  gyrogroup. Then
$a\boxplus b=b\oplus((\ominus b\oplus a)\oplus b)$
for all $a, b\in G$.
\end{proposition}

\begin{proof}
For all $a, b\in G$ we have
\begin{align*}
b\oplus((\ominus b\oplus a)\oplus b)&=(b\oplus(\ominus b\oplus a))\oplus \text{gyr}[b,\ominus b\oplus a]b \quad \text{by the left gyroassociative law~}
\\&=a\oplus \text{gyr}[a,\ominus b\oplus a]b \quad \text{by a left cancellation and a left loop property~}
\\&=a\oplus \text{gyr}[a,\ominus b]b \quad \text{by a right loop property~}
\\&=a\boxplus b.  \quad\quad\quad \text{by Definition \ref{defbox}~}
\end{align*}
\end{proof}

\begin{proposition}\cite[Theorem 2.22]{Ung}\label{the3.22}
Let $(G, \oplus)$ be a gyrogroup, and let $a\in G$. Then $L_a$ and $R_a$ are bijective.
\end{proposition}

\begin{proposition}\label{pro2.14s}
Let $G$ be a paratopological gyrogroup, $a\in G$ and $A\subseteq G$.

\begin{enumerate}
\item[(1)]  The right translation $R_{a}: G \rightarrow G$, where $R_{a}(x) = x \oplus a$ for every $x\in G$, is
homeomorphism;
\item[(2)] $A$ is closed if and only if $A\boxminus a$ is closed;
\item[(3)] $A$ is open if and only if $A\boxminus a$ is open.
\end{enumerate}
\end{proposition}
\begin{proof}
By Proposition \ref{the3.22} we have that $R_a$ is bijective.
To prove (1), let $x\in G$ and let $U$ be a
neighborhood of $R_{a}(x) = x\oplus a$. By the joint continuity of $G$, there exist open subsets
$U_a$, $U_x$ of $G$ such that $a\in U_a$, $x\in U_x$ and $U_x\oplus U_a\subseteq U$. Hence $R_{a}(U_x)=U_x\oplus a\subseteq
U_x\oplus U_a\subseteq U$, which shows that $R_{a}:G\rightarrow G$ is continuous.

To prove that $(R_{a})^{-1}$ is continuous, we put $y=x\oplus a$, then $x=y \boxminus a=y \boxplus (\ominus a)
=(\ominus a)\oplus((a\oplus y)\oplus (\ominus a))$ by Theorem \ref{the3.22s}.
That is $R^{-1}_{a}(x)=(\ominus a)\oplus((a\oplus x)\oplus (\ominus a))$ for any $x\in G$.
So, $R^{-1}_{a}=L_{\ominus a}\circ R_{\ominus a}\circ L_{a}$ which is continuous by Proposition \ref{pro23}.
Thus we get $R_{a}$ is homeomorphism.
Observe that $A\boxminus a=R^{-1}_{a}(A)$. (2) and (3) are true because $R^{-1}_{a}$ is homeomorphism.
\end{proof}

\begin{proposition}\cite{JX2}\label{pro23s}
Let $G$ be a paratopological gyrogroup, $x, y\in G$ and $A, B\subseteq G$.
\begin{enumerate}
\item[(1)] $\text{gyr}[x,y]: G \rightarrow G$, for every $x, y\in G$, is
homeomorphism;
\item[(2)] $A$ is closed if and only if $\text{gyr}[x,y](A)$ is closed;
\item[(3)] $A$ is open if and only if $\text{gyr}[x,y](A)$ is open.
\end{enumerate}
\end{proposition}

We characterize the families of subsets of a gyrogroup $G$ that can serve
as neighborhood bases of the identity 0 in some paratopological gyrogroups.
For $T_0$ paratopological gyrogroups, the following characterizations are evident.

\begin{theorem}\label{the1}
Let $G$ be a paratopological gyrogroup and let $\mathcal{U}$
be an open base at the identity element 0 in $G$. Then, the following properties hold:
\begin{enumerate}
\item[(1)] for every $U\in\mathcal{U}$, there exists an element $V\in \mathcal{U}$ such that $V\oplus V\subseteq U$;
\item[(2)] for every $U\in\mathcal{U}$, and every $x\in U$, there exists $V\in \mathcal{U}$ such that $x\oplus V\subseteq U$;
\item[(3)] for each $U\in\mathcal{U}$ and $x, y\in G$,
    there exists $V \in\mathcal{U}$ such that $\ominus(x\oplus y)\oplus ((V\oplus x)\oplus y)\subseteq U$
    and $V\subseteq \ominus(x\oplus y)\oplus ((U\oplus x)\oplus y)$;
\item[(4)] for $U, V\in\mathcal{U}$, there exists $W\in\mathcal{U}$ such that $W\subseteq U\cap V$.
\end{enumerate}
\end{theorem}
\begin{proof}
Let $U\in\mathcal{U}$.

(1) Let $x, y\in G$. For $G$ is a paratopological gyrogroup, then $op_2:G\times G\rightarrow G$ defined by $op_2(x,y)=x\oplus y$ is continuous.
Because $0\oplus 0\in U$, there exist $O, W\in\mathcal{U}$ such that $O\oplus W\subseteq U$.
Given that $\mathcal{U}$ is an open neighborhood of 0, there exists $V\in\mathcal{U}$ such that $V\subseteq O\cap W$.
Then $V\oplus V\subseteq W$.

(2) Let $x\in U$.
We define $L_x:G\rightarrow G$ by $L_x(y)=x\oplus y$.
Since $L_x(0)=x$ and $L_x$ is continuous at $0$, there exists $V\in\mathcal{U}$ such that $x\oplus V=L_x(V)\subseteq U$.

(3) Let $x, y\in G$. Since $G$ is a paratopological gyrogroup, the function $L_{\ominus(x\oplus y)}\circ R_y\circ R_x$ is continuous. Note that $L_{\ominus(x\oplus y)}\circ R_y\circ R_x(0)=\ominus(x\oplus y)\oplus ((0\oplus x)\oplus y)=0\in U$,
there exists $V_1\in\mathcal{U}$ such that $\ominus(x\oplus y)\oplus ((V_1\oplus x)\oplus y)\subseteq U$.
Note also that $\ominus(x\oplus y)\oplus ((U\oplus x)\oplus y)$ is an open
neighborhood of 0.
So there exists $V_2\in\mathcal{U}$ such that $V_2\subseteq \ominus(x\oplus y)\oplus ((U\oplus x)\oplus y)$.
Since $\mathcal{U}$ is an open base at 0, there exists $V\in\mathcal{U}$ such that $V\subseteq V_1\cap V_2$,
and the assertion
follows.

(4) It is clear since $\mathcal{U}$ is an open base at $0.$
\end{proof}

\begin{theorem}\label{the}
Let \( G \) be a gyrogroup, and let \( \mathcal{U} \) be a family of subsets of \( G \) that includes the identity element \( 0 \). Assume that \( \mathcal{U} \) satisfies the following conditions (1)-(5).
\begin{enumerate}
\item[(1)] for every $U\in\mathcal{U}$, there exists an element $V\in \mathcal{U}$ such that $V\oplus V\subseteq U$;
\item[(2)] for every $U\in\mathcal{U}$, and every $x\in U$, there exists $V\in \mathcal{U}$ such that $x\oplus V\subseteq U$;
\item[(3)] for each $U\in\mathcal{U}$ and $x, y\in G$,
    there exists $V \in\mathcal{U}$ such that $\ominus(x\oplus y)\oplus ((V\oplus x)\oplus y)\subseteq U$
    and $V\subseteq \ominus(x\oplus y)\oplus ((U\oplus x)\oplus y)$;
\item[(4)] for $U, V\in\mathcal{U}$, there exists $W\in\mathcal{U}$ such that $W\subseteq U\cap V$;
\item[(5)] for every $U\in\mathcal{U}$ and $x, y\in G$, we have \text{gyr}$[x, y]U\subseteq U$.
\end{enumerate}
\end{theorem}
Then the family $\mathcal{B}=\{a\oplus U:a\in G, U\in \mathcal{U}\}$ constitutes a base for a topology \( \mathcal{T} \) on \( G \).
With this topology, $G$ is a paratopological gyrogroup.

\begin{proof}
Let $\mathcal{U}$ be a family of subsets of $G$ such that conditions (1)-(5) hold.
Define
 $$\mathcal{T}=\{W\subseteq G: \text{for every~} x\in W~\text{there exists~} U\in\mathcal{U}~\text{such that~} x\oplus U\subseteq W\}.\quad\quad\quad (6)$$

{\bf Claim 1.} $\mathcal{T}$ is a topology on $G$.

It is clear that $G\in\mathcal{T}$ and $\emptyset\in\mathcal{T}$.
It also easy to see that $\mathcal{T}$ is closed under unions.
To show that $\mathcal{T}$ is closed under finite intersections, let $V,W\in\mathcal{T}$.
Let $x\in V\cap W$. Since $x\in V\in\mathcal{T}$ and $x\in W\in\mathcal{T}$, there exist $O,Q\in\mathcal{U}$ such that $x\oplus O\subseteq V$ and $x\oplus Q\subseteq W$.
From (4) it follows that there exists $U\in\mathcal{T}$ such that $U\subseteq O\cap Q$.
Then, we have $x\oplus U\subseteq V\cap W$.
Hence, $V\cap W\in\mathcal{T}$, and $\mathcal{T}$ is a topology on $G$.

{\bf Claim 2.} If $O\in \mathcal{U}$ and $g\in G$, then $g\oplus O\in\mathcal{T}$.

Take any $x\in g\oplus O$, then $\ominus g\oplus x\in O$.
By property (2), there exists $V\in\mathcal{U}$ such that $(\ominus g\oplus x)\oplus V\subseteq O$.
For $V$ and $\ominus g, x\in G$, we have $\text{gyr}[\ominus g,x]V\subseteq V$ by condition (5).
So we have $\ominus g\oplus(x\oplus V)=(\ominus g\oplus x)\oplus \text{gyr}[\ominus g,x]V\subseteq
(\ominus g\oplus x)\oplus V\subseteq O$,
that is $x\oplus V\subseteq g\oplus O$.
Hence $g\oplus O\in\mathcal{T}$.

{\bf Claim 3.} The family $\mathcal{B}=\{a\oplus U:a\in G, U\in \mathcal{U}\}$ is a base for the topology $\mathcal{T}$ on $G$.

The assertion is directly derived from the definition of $\mathcal{T}$ and the conditions specified in Claim 2.

{\bf Claim 4.} The multiplication in $G$ is continuous with respect to the topology $\mathcal{T}$.

By condition (3), for each $V\in\mathcal{U}$ and $x, y\in G$,
there exists $U_1 \in\mathcal{U}$ such that $\ominus(x\oplus y)\oplus ((U_1\oplus x)\oplus y)\subseteq V$
and $U_1\subseteq \ominus(x\oplus y)\oplus ((V\oplus x)\oplus y)$.
That is $((U_1\oplus x)\oplus y)\subseteq (x\oplus y)\oplus V$ and $ (x\oplus y)\oplus U_1\subseteq ((V\oplus x)\oplus y)$.
By setting \( y = 0 \), we obtain \( x \oplus U_1 \subseteq V\oplus x \).
Then there exists $U\in\mathcal{U}$ such that$(x \oplus U)\oplus y\subseteq(U_1\oplus x)
\oplus y\subseteq (x\oplus y)\oplus V$.

Let $x,y\in G$, and $O\in\mathcal{T}$ with $x\oplus y\in O$.
Then there exists $W\in\mathcal{U}$ such that $(x\oplus y)\oplus W\subseteq O$.
By condition (1) there exists $V\in\mathcal{U}$ such that $V\oplus V\subseteq W$.
By condition (3), there exists $U\in\mathcal{U}$ such that
$(x \oplus U)\oplus y\subseteq (x\oplus y)\oplus V$.
According to Claim 2, $x\oplus U, y\oplus V \in\mathcal{T}$.
A straightforward calculation demonstrates that
\begin{align*}
&(x \oplus U) \oplus (y \oplus V)
= \{(x \oplus u) \oplus (y \oplus v) \mid u \in U, v \in V\} \
\\&= \{((x \oplus u) \oplus y) \oplus \text{gyr}[x \oplus u, y](v)\mid u \in U, v \in V\} \quad \text{Left Gyroassociative Law}
\\&\subseteq \{((x \oplus u) \oplus y) \oplus v \mid u \in U, v \in V\}\quad\quad\quad\quad \text{by Condition(5)~}
\\&= ((x \oplus U) \oplus y) \oplus V \
\\&\subseteq ((x \oplus y) \oplus V) \oplus V \quad\quad\quad\quad \text{by Condition(3)~}
\\&= \{((x \oplus y) \oplus v_1) \oplus v_2 \mid v_1, v_2 \in V\} \
\\&= \{(x \oplus y) \oplus (v_1 \oplus \text{gyr}[v_1, x\oplus y](v_2)) \mid v_1, v_2 \in V\} \quad \text{Right Gyroassociative Law}
\\&\subseteq (x \oplus y) \oplus (V \oplus V) \quad\quad\quad\quad \text{by Condition(5)~}
\\&\subseteq (x \oplus y) \oplus W \quad\quad\quad\quad \text{by Condition(1)~}
\\&\subseteq O.
\end{align*}
This proves Claim 4.

We have proved that $G$ is a paratopological gyrogroup with the topology $\mathcal{T}$.
\end{proof}

\begin{corollary}
If \( G \) is a gyrogroup that satisfies Theorem \ref{the}, then \( G \) becomes a strongly paratopological gyrogroup with respect to the topology defined by Equation (6).
\end{corollary}
\begin{proof}
The family \( \mathcal{U} \) in Theorem \ref{the} is a basis of open sets at \( 0 \) that satisfies Definition \ref{defst}.
\end{proof}

\begin{lemma}\label{lem2.21}\cite{JX2}
Let the neighborhood base
$\mathcal{U}$ at 0 of $G$ witness that $G$ is a strongly paratopological gyrogroup.
Then we have
$(a\oplus U)\oplus W= a\oplus (U\oplus W)$ for each $a\in G, U, W\in \mathcal{U}.$
\end{lemma}

\begin{lemma}\label{lem19}\cite{JX2}
Let the neighborhood base
$\mathcal{U}$ at 0 of $G$ witness that $G$ is a strongly paratopological gyrogroup.
If $U\oplus V\subseteq W$, then $\ominus V\ominus U\subseteq \ominus W$, for each $W, U, V\in\mathcal{U}$.
\end{lemma}

Let's review the notion of coset spaces pertaining to a paratopological gyrogroup.

Consider a paratopological gyrogroup $(G, \tau, \oplus)$ with $H$ being an $L$-subgyrogroup of $G$.
As derived from \cite[Theorem 20]{Suk3}, the set $G/H = \{a \oplus H : a \in G\}$ constitutes a division of $G$ into distinct parts. The function $\pi$, which maps each element $a\in G$ to the coset $a \oplus H$ onto $G/H$, ensures that the preimage $\pi^{-1}(\pi(a))=a \oplus H$. We refer to $\tau(G)$ as the topology on $G$. In defining a topology $\widetilde{\tau}$ on the set of left cosets $G/H$ of the gyrogroup $G$, we employ $\widetilde{\tau}=\{O \subseteq G/H : \pi^{-1}(O)\in \tau(G)\}$, with $\tau(G)$ indicating the established topology of $G$.

\begin{proposition}\label{pro2.28}\cite{JX2}
Let $(G, \tau,\oplus)$ be a paratopological gyrogroup and $H$ a $L$-subgyrogroup of $G$.
Then the natural homomorphism
$\pi$ from a paratopological gyrogroup $G$ to its quotient topology on $G/H$ is an open and continuous mapping.
\end{proposition}

\section{The Hausdorff number of a paratopological gyrogroup}

Consider $G$ as a Hausdorff paratopological gyrogroup with the identity 0. For any element $x\in G$, distinct from 0, there exists an open neighborhood $V$ of 0 such that $V \cap (x \oplus V)= \emptyset$. This implies that $x\notin V\boxminus V$. Consequently, $\bigcap_{U\in\mathcal{U}} (U\boxminus U)=\{0\}$, where $\mathcal{U}$ denotes the family of open neighborhoods of the identity 0 in $G$.
Motivated by this observation, we define the {\it Hausdorff number} of $G$, denoted by $Hs(G)$, as the
minimum cardinal number $\kappa$ such that for every neighbourhood $U$ of 0 in $G$,
there exists a family $\gamma$ of neighbourhoods of 0 such that $\bigcap_{V\in\gamma} (V\boxminus V)\subseteq U$ and $|\gamma|\leq\kappa$.
A paratopological gyrogroup $G$ is $T_1$ if and only if for each $x\in G\setminus\{0\}$, there exists $V\in\mathcal{U}$ such that $x\notin V$ or, equivalently, $\bigcap_{V\in\gamma} (\ominus V)=\{0\}$. Motivated by this observation, we define the {\it weakly Hausdorff number} of a $T_1$ paratopological gyrogroup $G$, denoted by $wHs(G)$, as the minimum cardinal number $\kappa$ such that for for every neighborhood $U$ of 0 in $G$ there exists a family $\gamma$ of neighborhoods of 0 such that $\bigcap_{V\in\gamma} (\ominus V)\subseteq U$ and $|\gamma|\leq\kappa$.

In \cite{JX2}, the authors proved the following conclusion:
\begin{lemma}\cite[Lemma 3.4]{JX2}\label{pro3.1}
Let the neighborhood base
$\mathcal{U}$ at 0 of $G$ witness that $G$ is a strongly paratopological gyrogroup.
Then $G$ is a strongly topological gyrogroup if and only if the inverse operation is continuous at the identity 0.
\end{lemma}

By the provided definition, one can deduce that a strongly paratopological gyrogroup $G$ qualifies as a strongly topological gyrogroup precisely when $wHs(G)=1$.

\begin{proposition}\label{pro3.2}
If $G$ is a Hausdorff paratopological gyrogroup, then the inequality $wHs(G)\leq Hs(G)$ holds.
\end{proposition}
\begin{proof}
For any $V\in \mathcal{U}$ and $v\in V$ we can get $0\boxminus v=\ominus v\oplus((v\oplus 0)\ominus v)=\ominus v$
by Proposition \ref{the3.22s}.
That is $\ominus V\subseteq V\boxminus V$. Hence we have $wHs(G)\leq Hs(G)$ by Definitions of Hausdorff number and weakly Hausdorff number.
\end{proof}

\begin{proposition}\label{pro3.2s1}
If $K$ is a subgyrogroup of a Hausdorff paratopological gyrogroup $G$, then the inequality $Hs(K)\leq Hs(G)$ holds.
\end{proposition}
\begin{proof}
Let \( K \) be a subgyrogroup of \( G \).
Clearly, \( K\) is a paratopological gyrogroup with an open base \( \mathcal{U}_{K} = \{ U \cap K \mid U \in \mathcal{U} \} \) at \( 0\).
For any $V\in \mathcal{U}$ we can get $(V\cap K)\boxminus (V\cap K)\subseteq V\boxminus V$.
Hence we have $Hs(K)\leq Hs(G)$ by the Definition of Hausdorff number.
\end{proof}

\begin{proposition}\label{pro3.3s1}
Let $\Pi=\prod_{i\in I} G_i$ be the topological product of a family of Hausdorff paratopological gyrogroups such that
$Hs(G_i)\leq\kappa$ for each $i\in I$. Then $Hs(\Pi)\leq\kappa$.
\end{proposition}

\begin{proof}
Since every canonical open neighborhood in \( \Pi \) depends on finitely many coordinates,
we will prove that $Hs(\prod_{i\in I} G_i)\leq\kappa$ for $I=\{1, 2\}$.

Let \(G\) and \(H\) be two Hausdorff paratopological gyrogroups with
$Hs(G), Hs(H)\leq\kappa$.
For every neighborhoods \(U_G\) and \(U_H\) of the identity in \(G\) and \(H\) respectively,
we have $(U_G\times U_H)\boxminus(U_G\times U_H)=\{(u_G, u_H)\boxminus(v_G, v_H)|
u_G, v_G\in G; u_H, v_H\in H\}=\{(u_G \boxminus v_G, u_H \boxminus v_H)|
u_G, v_G\in G; u_H, v_H\in H\}$.
Hence we have $Hs(G \times H)\leq \kappa$ by the Definition of Hausdorff number.
\end{proof}

\begin{proposition}\label{pro3.4s1}
Every first-countable Hausdorff paratopological gyrogroup $G$ satisfies $Hs(G)\leq\omega$.
\end{proposition}

\begin{proof}
Let $\{V_n: n \in\omega\}$ be a countable neighborhood base of the identity 0 in $G$.
Take any open neighbourhood $U$ of 0.
Since $G$ is Hausdorff,
for any element $x\in G$ distinct from 0, there exists an open neighborhood $V_n$ of 0 such that $V_n\cap (x \oplus V_n)= \emptyset$. This implies that $x\notin V_n\boxminus V_n$. Consequently, $\bigcap_{n \in\omega} (V_n\boxminus V_n)=\{0\}\subseteq U$, that is $Hs(G)\leq\omega$.
\end{proof}


\begin{proposition}\label{pro3.2s}
If $K$ is a subgyrogroup of a $T_1$ paratopological gyrogroup $G$, then the inequality $wHs(K)\leq wHs(G)$ holds.
\end{proposition}
\begin{proof}
Let \( K \) be a subgyrogroup of \( G \).
Clearly, \( K\) is a paratopological gyrogroup with an open base \( \mathcal{U}_{K} = \{ U \cap K \mid U \in \mathcal{U} \} \) at \( 0\).
For any $V\in \mathcal{U}$ we can get $\ominus(V\cap K)\subseteq \ominus V$.
Hence we have $wHs(K)\leq wHs(G)$ by the Definition of weakly Hausdorff number.
\end{proof}

\begin{proposition}\label{pro3.3s}
Let $\Pi=\prod_{i\in I} G_i$ be the topological product of a family of $T_1$ paratopological gyrogroups such that
$wHs(G_i)\leq\kappa$ for each $i\in I$. Then $wHs(\Pi)\leq\kappa$.
\end{proposition}
\begin{proof}
Since every canonical open neighborhood in \( \Pi \) depends on finitely many coordinates,
we will prove that $wHs(\prod_{i\in I} G_i)\leq\kappa$ for $i=1, 2$.

Let \(G\) and \(H\) be two $T_1$ paratopological gyrogroups with
$wHs(G), wHs(H)\leq\kappa$.
For every neighborhoods \(U_G\) and \(U_H\) of the identity in \(G\) and \(H\) respectively,
we have $\ominus(U_G \times U_H)=\{\ominus(u_G, u_H)|
u_G\in G; u_H\in H\}=\{(\ominus u_G , \ominus u_H)|
u_G\in G; u_H\in H\}$.
Hence we have $wHs(G \times H)\leq \kappa$ by the Definition of weakly Hausdorff number.
\end{proof}

\begin{proposition}\label{pro3.3}
For every $T_1$ paratopological gyrogroup $G$, the inequalities $wHs(G)\leq\psi(G)\leq \chi(G)$ hold.
\end{proposition}
\begin{proof}
Based on the definition of pseudocharacter, character and weakly Hausdorff number, the result can be directly obtained.
\end{proof}

For a Hausdorff paratopological gyrogroup $G$,
the Hausdorff number $Hs(G)$ is bounded by the Lindel\"{o}f number $l(G)$ of
the gyrogroup:

\begin{proposition}\label{pro3.4s}
Let $G$ be a Hausdorff paratopological gyrogroup. Then it holds that $Hs(G) \leq l(G)$.
\end{proposition}
\begin{proof}
Let \( U \) be any neighborhood of the identity element \(0 \) in \( G \).
Since \( G \) is a Hausdorff paratopological gyrogroup, for every  \( x\in G\setminus U \), there exists an open neighborhood \( V_x \) of \( 0 \) such that \( V_x \cap (x\oplus V_x)\oplus V_x=\emptyset \),
which means $v_1\neq(x\oplus v_2)\oplus v_3$ for each $v_1, v_2, v_3\in V_x$.
This is equivalent to
$v_1\boxminus v_3 \neq x \oplus v_2$.
Or, equivalently,
$(V_x\boxminus V_x) \cap (x \oplus V_x)= \emptyset$.
Given that \( G \setminus U \) is closed in \( G \) and the family \( \{x\oplus V_x : x \in G \setminus U\} \) covers \( G \setminus U \), we can find a subset \( A \subseteq G \setminus U \) such that \( G \setminus U \subseteq \bigcup_{x \in A} x\oplus V_x \) and \( |A|\leq l(G)\). Consequently, the family \( \gamma = \{V_x : x \in A\} \) satisfies \( \bigcap_{x \in A} V_x\boxminus V_x\subseteq U \) and \( |\gamma| \leq l(G) \). Therefore, this leads to the inequality
\(Hs(G) \leq l(G).\)
\end{proof}

It has been shown that the weakly Hausdorff number \( wHs(G) \) of a paratopological gyrogroup \( G \) is less than or equal to the Lindel\"{o}f number \( l(G) \) of the gyrogroup.

\begin{proposition}\label{pro3.4s}
Let $G$ be a $T_1$ paratopological gyrogroup. Then it holds that $wHs(G) \leq l(\overline{\ominus V})$ for any nonempty open subset $V$ of $G$.
\end{proposition}
\begin{proof}
Since $G$ is homogeneous,
it is enough to show that the inequality $wHs(G)\leq l(\overline{\ominus V})$  holds for every open neighborhood $V$ of the identity 0 in $G$. Put $\tau =
l(\overline{\ominus V})$ and
take any open neighborhood $U$ of the identity 0 in $G$.
Clearly $l(\overline{\ominus V}\setminus U)\leq\tau$.
Since $G$ is a $T_1$ paratopological gyrogroup, for each $x\in \overline{\ominus V}\setminus U$ there exists $W_x$ of open neighborhood at the identity 0 such
that $0\notin (x\oplus W_x)\oplus W_x$,
which means $(x\oplus w_1)\oplus w_2\neq 0$ for each $w_1, w_2\in W_x$.
This is equivalent to
$\ominus w_2 \neq (x \oplus w_1)$.
Or, equivalently,
$\ominus W_x \cap (x \oplus W_x)=\emptyset$.
The set $\overline{\ominus V}\setminus U$ is closed in $G$ and the family $\{x\oplus W_x:x\in \overline{\ominus V}\setminus U\}$ is
an open cover of $\overline{\ominus V}\setminus U$, so there exists a family $\gamma=\{W_{\alpha}:\alpha<\tau\}$ of open neighborhoods at the identity 0 and a subset $A=\{a_{\alpha}:\alpha<\tau\}\subseteq \overline{\ominus V}\setminus U$ such that $\overline{\ominus V}\setminus U\subseteq \bigcup_{\alpha<\tau}(a_{\alpha}\oplus W_{\alpha})$.
Put $\delta=\gamma \bigcup\{V\}$. It follows that $\bigcap_{W\in\delta}(\ominus W)\subseteq U$.
It implies that $wHs(G)\leq \tau=l(\overline{\ominus V})$.
\end{proof}

According to Proposition \ref{pro3.4s} the following corollary is obvious.
\begin{corollary}\label{cor3.7}
For any $T_1$ paratopological gyrogroup $G$, the inequality $wHs(G) \leq l(G)$ holds true.
\end{corollary}

Let $G$ be a paratopological group.
For every neighbourhood $U$ of the identity $e$ in $G$, if one
can find a countable family $\gamma$ of neighbourhoods of $e$ such that for each $x\in G$ there exists $V\in\gamma$ satisfying
$xVx^{-1}\subseteq U$, then $G$ is called $\omega$-balanced.
Such a family $\gamma$ is usually called subordinated to $U$.
Given that gyrogroups represent a natural generalization of groups, we extend the notion of $\omega$-balanced to paratopological gyrogroups.


For a paratopological gyrogroup $G$, a family $\gamma$ of open neighborhoods of 0 in $G$ is called {\it subordinated} to a neighborhood $U$ of the identity in $G$ if the following conditions are satisfied:
\begin{enumerate}
\item[(1)] for each $x, y\in G$, there exists $V \in\gamma$ such that $\ominus(x\oplus y)\oplus ((V\oplus x)\oplus y)\subseteq U$;
\item[(2)] for each $x, y\in G$, there exists $V \in\gamma$ such that $V\subseteq \ominus(x\oplus y)\oplus ((U\oplus x)\oplus y)$.
\end{enumerate}

A paratopological gyrogroup $G$ is called $\omega$-{\it balanced} if for every neighborhood $U$ of the identity 0 in $G$, one can find a countable family $\gamma$ of open neighborhoods of 0 in $G$
such that $\gamma$ is subordinated to $U$.

\begin{remark}
Given that when \text{gyr}$[x, y]$ acts as the identity map for all $x, y$ in $G$, the gyrogroup is a group, and at this time, the definition of $\omega$-balanced in the paratopological gyrogroup is consistent with the definition of $\omega$-balanced in the paratopological group.
\end{remark}

\begin{proposition}\label{pro3.12s}
Every subgyrogroup of an $\omega$-balanced paratopological gyrogroup is $\omega$-balanced.
\end{proposition}
\begin{proof}
Let \( G \) be an \(\omega\)-balanced paratopological gyrogroup with a countable family \(\gamma\) of open neighborhoods of 0 in \( G \) such that \(\gamma\) is subordinated to every neighborhood \( U \) of the identity element 0. Let \( H \) be a subgyrogroup of \( G \).
Clearly, \( H \) is a paratopological gyrogroup with an open base \( \mathcal{U}_{H} = \{ U \cap H \mid U \in \mathcal{U} \} \) at \( 0\).
Let \( x, y \in H \) and let \( U \in \mathcal{U}\). Then, there exists $V \in\gamma$ such that
$\ominus(x\oplus y)\oplus (((V\cap H)\oplus x)\oplus y)=(\ominus(x\oplus y)\oplus ((V\oplus x)\oplus y))\cap H\subseteq U\cap H$.
According to the same operation, it follows that $V\cap H\subseteq\ominus(x\oplus y)\oplus (((U\cap H)\oplus x)\oplus y)$.
Thus, the countable family $\gamma_H=\{ V \cap H \mid V \in \gamma\}$ is subordinated to $U\cap H$.
\end{proof}

\begin{theorem}\label{the3.19}
Every first-countable paratopological gyrogroup $G$ is $\omega$-balanced.
\end{theorem}

\begin{proof}
Let $\{V_n: n \in\omega\}$ be a countable base of the space $G$ at the identity 0
of the gyrogroup $G$. Take any open neighbourhood $U$ of 0.
Since the map $f=L_{\ominus(x\oplus y)}\circ R_y\circ R_x: G \rightarrow G$ is continuous and $f(0)=0$, there exists $n_1 \in\omega$ such that
$f(V_{n_1})\subseteq U$. It follows that $\ominus(x\oplus y)\oplus ((V_{n_1}\oplus x)\oplus y)\subseteq U$.
It is obvious that $\ominus(x\oplus y)\oplus ((U\oplus x)\oplus y)$ is an open neighbourhood of 0, so there exists
$n_2 \in\omega$ such that $V_{n_2}\subseteq\ominus(x\oplus y)\oplus ((U\oplus x)\oplus y)$.
Put $V_n\subseteq V_{n_1}\cap V_{n_2}$. Then $V_n$ is the one we want.
Hence, $G$ is $\omega$-balanced.
\end{proof}

\begin{proposition}\label{pro3014s}
The topological product \(\Pi= \prod_{i \in I} G_i \) of a family of $\omega$-balanced paratopological gyrogroups is also $\omega$-balanced.
\end{proposition}
\begin{proof}
Since every canonical open set in \( \Pi \) depends on finitely many coordinates,
we will prove that the product of two \(\omega\)-balanced paratopological gyrogroups is \(\omega\)-balanced.

Let \(G\) and \(H\) be two \(\omega\)-balanced paratopological gyrogroups. We need to prove that \(G \times H\) is also \(\omega\)-balanced.
Since \(G\) and \(H\) are \(\omega\)-balanced, by definition,
for every neighborhoods \(U_G\) and \(U_H\) of the identity in \(G\) and \(H\) respectively, one can find countable families $\gamma_G, \gamma_H$ of open neighborhoods of the identity in \(G\) and \(H\)
such that for each $x_1, y_1\in G$, $x_2, y_2\in H$ there exists $V_G \in\gamma_G$ and $V_H\in\gamma_H$ such that
\[ \ominus(x_1\oplus y_1)\oplus ((V_G\oplus x_1)\oplus y_1)\subseteq U_G, \]
\[V_G\subseteq \ominus(x_1\oplus y_1)\oplus ((U_G\oplus x_1)\oplus y_1),\]
\[ \ominus(x_2\oplus y_2)\oplus ((V_H\oplus x_2)\oplus y_2)\subseteq U_H, \]
\[V_H\subseteq \ominus(x_2\oplus y_2)\oplus ((U_H\oplus x_2)\oplus y_2).\]

Consider the neighborhood \(U_G \times U_H\) in \(G \times H\). Clearly:

$\ominus((x_1, x_2)\oplus (y_1, y_2))\oplus (((V_G\times V_H)\oplus (x_1, x_2))\oplus (y_1, y_2))
=(\ominus(\oplus y_1)\oplus ((V_G\oplus x_1)\oplus y_1))\times(\ominus(x_2\oplus y_2)\oplus ((V_H\oplus x_2)\oplus y_2)) \subseteq U_G \times U_H$,

$V_G\times V_H\subseteq \ominus((x_1\times x_2)\oplus (y_1\times y_2))\oplus (((U_G \times U_H)\oplus (x_1\times x_2))\oplus (y_1\times y_2))$.
This implies that \(G \times H\) is \(\omega\)-balanced.
\end{proof}

The verification of Theorems \ref{the3.26}, \ref{the3.30s} and \ref{the3.30} necessitates the utilization of the next two lemmas.

A subset $V$ of a semitopological gyrogroup $G$ is defined as $\omega$-{\it good} if there exists a countable subfamily $\gamma$ of the neighborhood system $\mathcal{U}$ of the identity element 0, such that for each $x \in V$, there exists $W \in \gamma$ satisfying the condition $x \oplus W \subseteq V$.

\begin{lemma}\label{lem3.16}
For every strongly paratopological gyrogroup $H$, there exists a local base at the identity element comprising of $\omega$-good sets.
\end{lemma}
\begin{proof}
Suppose that the neighborhood base $\mathcal{U}$ at 0 of $G$ witness that $G$ is a strongly paratopological gyrogroup.
Consider $U\in\mathcal{U}$. Select a countable collection $\{U_n: n \in \omega\}\subseteq\mathcal{U}$, ensuring that $U_0=U$
and $U_{n+1}\oplus U_{n+1}\subseteq U_n$, for each $n\in\omega$.
Set $V_1=U_1$ and $V_2=U_1\oplus U_2$. Extending this pattern, for any given $n\in\omega$,
let $V_{n} = V_{n-1} \oplus U_{n}$ for any $n>1$.
It is evident that $V_n$ is open and $V_n\subseteq V_{n+1}$ for all natural numbers $n$.
By applying Lemma \ref{lem2.21} and using a typical reasoning process, we can conclude that $((U_{k+1}\oplus U_{k+2})\oplus U_{k+3})\oplus \cdot\cdot\cdot \oplus U_{k+n+1}\subseteq U_k$, for all $k,n\in\omega$.
Specifically, we have that $V_n\subseteq U_0=U$, for every integer $n \geq 1$. Consequently, we can deduce that the open set $V=\bigcup_{n=1}^{\infty}V_n$ is also contained within $U$.
Considering an element $x\in V$, based on the definition of $V$, it means that $x\in V_n$, for some $n\geq 1$. Given that $V_n\oplus U_{n+1}= V_{n+1} \subseteq V$, it follows that $x\oplus U_{n+1}\subseteq V$.
Alternatively, the family $\{U_n: n\in\omega\}$ provides evidence that the set $V$ meets the criteria to be considered $\omega$-good.
\end{proof}


\begin{lemma}\label{lem3.17}
Let $G$ be a paratopological gyrocommutative gyrogroup.
Assume that a family $\gamma\subseteq \mathcal{U}$ fulfills the following condition:
\begin{enumerate}
\item[(a)] for each $U \in \gamma$, there exists $V \in \gamma$ such that $V\oplus V\subseteq U$;
\item[(b)] for each $U \in \gamma$ and $x, y\in G$, there exists $V \in\gamma$ such that $\text{gyr}[x,y]V\subseteq U$;
\item[(c)] $\gamma$ is subordinated to $U$, for each $U\in\gamma$.
\end{enumerate}
In that case, the set $N=\bigcap\{U\cap(\ominus U): U\in\gamma\}$ forms an invariant subgyrogroup of $G$.
Furthermore, if the following condition is also met:
\begin{enumerate}
\item[(d)] $\bigcap\{\ominus V: V\in\gamma\}\subseteq U$, for all $U \in \gamma$,
\end{enumerate}
then $N=\bigcap\{U: U \in \gamma\}$ and the subgyrogroup $N$ is closed in $G$.
\end{lemma}
\begin{proof}
 Assuming that conditions (a)-(c) are satisfied. It is evident that $N =\ominus N$.
Indeed, for any $a, b\in N$ and $U\in \gamma$, we can choose $V \in\gamma$ such that $\text{gyr}[a,b]V\subseteq U$,
and for $V$, there exists $W\in \gamma$ such that $W\oplus W \subseteq V$.
In this case, it follows that $a, b\in W\cap(\ominus W)$.
So we can get $a\oplus b\in W\oplus W\subseteq V\subseteq U$. Additionally by applying Theorem \ref{the1.3} (7), $a\oplus b=\ominus\text{gyr}[a,b](\ominus a\ominus b)\in \ominus\text{gyr}[a,b](W\oplus W)\subseteq \ominus\text{gyr}[a,b]V
\subseteq\ominus U$.
 Therefore, $a\oplus b\in U\cap(\ominus U)$.
Thus we can conclude that $N$ is a subgyrogroup of $G$.

We intend to demonstrate that $N$ is an invariant subgyrogroup of $G$.
Take $h\in N$, $x, y\in G$ and $U \in\gamma$.
By (b), we can find $V \in\gamma$ such that $\text{gyr}[x,y]V\subseteq U$.
Since $h\in N$, it follows that $h \in V\cap(\ominus V)$.
Hence $\text{gyr}[x,y]h\in\text{gyr}[x,y]V\subseteq U$ and $\text{gyr}[x,y]h\in\text{gyr}[x,y](\ominus V)=\ominus\text{gyr}[x,y]V\subseteq \ominus U$.
This implies that $\text{gyr}[x,y]h\in U\cap(\ominus U)$.
Moreover, it implies that $\text{gyr}[x,y](N)\subseteq N$.

By (c), we can find $V \in\gamma$ such that
$\ominus(x\oplus y)\oplus ((V\oplus x)\oplus y)\subseteq U$ and $V\subseteq \ominus(x\oplus y)\oplus ((U\oplus x)\oplus y)$.
Because $\ominus x$ and $\ominus y$ are also elements of $G$,
we have $\ominus(\ominus x\oplus (\ominus y))\oplus ((V\oplus (\ominus x))\oplus (\ominus y))\subseteq U$.
And since $G$ is a gyrocommutative gyrogroup, it can get $(x\oplus y)\oplus ((V\ominus x)\ominus y)=\ominus(\ominus x\oplus (\ominus y))\oplus ((V\oplus (\ominus x))\oplus (\ominus y))\subseteq U$ by Theorem \ref{the2.5com} (1).
Since $h\in N$, it follows that $h \in V\cap(\ominus V)$.
Hence $\ominus(x\oplus y)\oplus ((h\oplus x)\oplus y)\in\ominus(x\oplus y)\oplus ((V\oplus x)\oplus y)\subseteq U$
and
$\ominus(x\oplus y)\oplus ((h\oplus x)\oplus y)\in\ominus(x\oplus y)\oplus ((\ominus V\oplus x)\oplus y)=\ominus((x\oplus y)\oplus ((V\ominus x)\ominus y))\subseteq \ominus U$.
This implies that $\ominus(x\oplus y)\oplus ((h\oplus x)\oplus y)\in U\cap(\ominus U)$.
That is $\ominus(x\oplus y)\oplus ((N\oplus x)\oplus y)\subseteq N$.
The same method yields $N\subseteq \ominus(x\oplus y)\oplus ((N\oplus x)\oplus y)$.
Consequently, by Theorem \ref{pro2.11s} we can deduce that $N$ is an invariant subgyrogroup.

If conditions (a)-(d) are satisfied, then $\bigcap_{U\in\gamma} (\ominus U)=\bigcap_{U\in\gamma}U$. This leads to $N = \bigcap_{U\in\gamma} U$.
To show that $N$ is closed in $G$, we consider $x\in G\setminus N$. We select $U\in\gamma$ such that $x\notin \ominus U$. Additionally, let $V\in\gamma$ such that $V\oplus V\subseteq U$ and we can get $\ominus V\oplus(\ominus V)\subseteq\ominus U$
by Theorem \ref{the2.5com} (1).
We claim that $(V\oplus x)\cap N =\emptyset$. Suppose otherwise, $x \in \ominus V\oplus N \subseteq
\ominus V\oplus(\ominus V)\subseteq\ominus U$. This contradicts the choice of $U$.
Hence, our claim holds, and we can conclude that $(V\oplus x)\cap N =\emptyset$. Therefore, $N$ is an closed invariant subgyrogroupa of $G$.
\end{proof}

\begin{lemma}\label{lem3.17s}
Let $G$ be a paratopological gyrogroup.
Assume that a family $\gamma\subseteq \mathcal{U}$ fulfills the following condition:
\begin{enumerate}
\item[(a)] for each $U \in \gamma$, there exists $V \in \gamma$ such that $V\oplus V\subseteq U$;
\item[(b)] for each $U \in \gamma$ and $x, y\in G$, there exists $V \in\gamma$ such that $\text{gyr}[x,y]V\subseteq U$;
\item[(c)] $\gamma$ is subordinated to $U$, for each $U\in\gamma$;
\item[(d)] $\bigcap\{\ominus V: V\in\gamma\}\subseteq U$, for all $U \in \gamma$.
\end{enumerate}
In that case, the set $N=\bigcap\{U\cap(\ominus U): U\in\gamma\}
=\bigcap\{U: U \in \gamma\}$ forms an invariant subgyrogroup of $G$.
\end{lemma}
\begin{proof}
If conditions (a)-(d) are satisfied, then $\bigcap_{U\in\gamma} (\ominus U)=\bigcap_{U\in\gamma}U$.
It is evident that $N =\ominus N$.
Take $a, b\in N$ and $U \in\gamma$.
By (a), we can find $V \in\gamma$ such that $V\oplus V\subseteq U$,
then $a, b\in V$.
It follows that $a\oplus b\in V\oplus V\subseteq U$.
So we can conclude that $N$ is a subgyrogroup of $G$.

We intend to demonstrate that $N$ is an invariant subgyrogroup of $G$.
Take $h\in N$, $x, y\in G$ and $U \in\gamma$.
By (b), we can find $V \in\gamma$ such that $\text{gyr}[x,y]V\subseteq U$.
Since $h\in N$, it follows that $h \in V$.
Hence $\text{gyr}[x,y]h\in\text{gyr}[x,y]V\subseteq U$.
This implies that $\text{gyr}[x,y](N)\subseteq N$.

By (c), we can find $V \in\gamma$ such that
$\ominus(x\oplus y)\oplus ((V\oplus x)\oplus y)\subseteq U$ and $V\subseteq \ominus(x\oplus y)\oplus ((U\oplus x)\oplus y)$.
Since $h\in N$, it follows that $h \in V$.
Hence $\ominus(x\oplus y)\oplus ((h\oplus x)\oplus y)\in\ominus(x\oplus y)\oplus ((V\oplus x)\oplus y)\subseteq U$
and
$h\in V\subseteq \ominus(x\oplus y)\oplus ((U\oplus x)\oplus y)$.
This implies that
$\ominus(x\oplus y)\oplus ((N\oplus x)\oplus y)\subseteq N$ and $N\subseteq \ominus(x\oplus y)\oplus ((N\oplus x)\oplus y)$.
Consequently, by Theorem \ref{pro2.11s} we can deduce that $N$ is an invariant subgyrogroup.
\end{proof}

Given a topological property $\mathcal{P}$, we say that a paratopological (topological) gyrogroup $G$
is {\it projectively} $\mathcal{P}$ if for every neighborhood $U$ of the identity in $G$, there exists a continuous homomorphism
$p: G\rightarrow H$ onto a paratopological (topological) gyrogroup $H$ with property $\mathcal{P}$ such that $p^{-1}(V)\subseteq U$, for some neighborhood $V$ of the identity in $H$.

The following result stems from the definition of the product topology.

\begin{proposition}\label{the319}
Let \(\mathcal{P}\) represent any class of paratopological (topological) gyrogroups that is closed under finite products. Let \(H\) denote the topological product of a family \(\{H_\alpha : \alpha \in A\}\) of gyrogroups belonging to the class \(\mathcal{P}\). Then every subgyrogroup of \(H\) is projectively $\mathcal{P}$.
\end{proposition}

\begin{theorem}\label{the319s}
Let \(\mathcal{P}\) be a class of paratopological (topological) gyrogroups, and let \(\tau\) be an infinite cardinal number. Suppose \(G\) is a paratopological (topological) gyrogroup that is projectively $\mathcal{P}$ and has a base \(\mathcal{B}\) of open neighborhoods of the identity such that \(|\mathcal{B}|\leq\tau\). Then \(G\) is topologically isomorphic to a subgyrogroup of the product of a family \(\{H_\alpha : \alpha \in A\}\) of gyrogroups, where each \(H_\alpha\) belongs to the class \(\mathcal{P}\) and \(|A|\leq\tau\).
\end{theorem}

\begin{proof}
Let \(\mathcal{B}\) be a base of open neighborhoods of the identity in the topological gyrogroup \(G\), with the property that \(|\mathcal{B}|\leq\tau\).
For each \(U \in \mathcal{B}\), we choose a continuous homomorphism \(f_U : G \to H_U\), where \(H_U \in \mathcal{P}\), such that \(f_U^{-1}(V) \subseteq U\) for some open neighborhood \(V\) of the identity in \(H_U\). The diagonal product \(h\) of the family \(\{f_U : U \in \mathcal{B}\}\) establishes a topological isomorphism of \(G\) onto a topological subgyrogroup of the topological product of the family \(\{H_U : U \in \mathcal{B}\}\). Specifically, for any \(x \in G\) and \(U \in \mathcal{B}\), the \(U\)-th coordinate of \(h(x)\) corresponds to \(f_U(x)\) in \(H_U\).
\end{proof}

\begin{theorem}\label{the3.18s}
A $T_i$ paratopological gyrogroup $G$ is projectively first-countable $T_i$ if and only if it admits a homeomorphic embedding as a subgyrogroup into a product of first-countable $T_i$ paratopological gyrogroups, for $i= 0, 1, 2, 3, 3\frac{1}{2}$.
\end{theorem}

\begin{proof}
We can directly derive the conclusion from Proposition \ref{the319} and Theorem \ref{the319s}.
\end{proof}

In the subsequent results, we provide a characterization of projectively first-countable $T_0$ (strongly) paratopological gyrogroups.

\begin{theorem}\label{the3.26}
A strongly paratopological gyrocommutative gyrogroup
$G$ is topologically isomorphic to a subgyrogroup of a topological product
of first-countable $T_0$ strongly paratopological gyrocommutative gyrogroups if and only if $G$ is $T_0$ and $\omega$-balanced.
\end{theorem}

\begin{proof}
Suppose that $G$ is a subgyrogroup of the product $\Pi=\prod_{i\in I}H_i$
of a family of first-countable paratopological gyrogroups.
By Theorem \ref{the3.19} every first-countable paratopological gyrogroup $H_i$ is $\omega$-balanced.
Then the product gyrogroup $\Pi$ is $\omega$-balanced by Proposition \ref{pro3014s}.
Since by Proposition \ref{pro3.12s}, the property of being $\omega$-balanced is hereditary with respect to taking subgyrogroup, we conclude that $G$ is $\omega$-balanced as well.

 By Theorem \ref{the3.18s}, to prove the sufficiency, it suffices to verify that every $\omega$-balanced strongly paratopological gyrogroup
$G$ is projectively
$T_0$ first-countable.
Let the neighborhood base
$\mathcal{U}$ at $0_G$ of $G$ witness that $G$ is a strongly paratopological gyrogroup.
Take an arbitrary $U_0\in\mathcal{U}$.
 We shall define a continuous homomorphism $p: G\rightarrow H$
 onto a first-countable $T_0$ strongly paratopological gyrocommutative gyrogroup $H$
and find a neighbourhood $V_0$ of
the identity in $H$ such that $p^{-1}(V_0)\subseteq U_0$.
To achieve this, we will construct by induction a countable family $\gamma$ of open neighborhoods of the identity $0_G$ in $G$. We then define $p$ as the natural homomorphism from $G$ onto the quotient gyrogroup $G/N$, where $N$ is given by $N = \bigcap\{U \cap (\ominus U) : U \in \gamma\}$. Consequently, a local base at the identity of $H = G/N$ will be the family $\{p(V) : V \in \gamma\}$.

Let $\mathcal{U}^*(0_G)$ represent the subfamily of $\mathcal{U}$ that consists of $\omega$-good sets. According to Lemma \ref{lem3.16}, $\mathcal{U}^*(0_G)$ serves as a local base for $G$ at $0_G$. By induction, we aim to construct a sequence denoted as $\{\gamma_n: n\in\omega\}$, where each $\gamma_n\subseteq \mathcal{U}^*(0_G)$ for all $n\in\omega$.

Choose $U^*_0\in\mathcal{U}^*(0_G)$ such that $U^*_0 \subseteq U_0$. Define $\gamma_0 = \{U^*_0\}$. Assume that for a given $n \in \omega$, the families $\gamma_0, \ldots, \gamma_n$ have been defined. These families are constructed to meet the following conditions for each index $k \leq n$:

\begin{enumerate}
\item[(i)] $\gamma_k \subseteq \mathcal{U}^*(0_G)$ and $|\gamma_k|\leq\omega$;
\item[(ii)] $\gamma_{k-1} \subseteq\gamma_k$;
\item[(iii)] $\gamma_k$ is closed under finite intersections;
\item[(iv)] for every $U\in \gamma_{k-1}$, there exists $V \in\gamma_{k}$ such that $V\oplus V\subseteq U$;
\item[(v)] the family $\gamma_k$ is subordinated to $U$, for each $U\in \gamma_{k-1}$;
\item[(vi)] for every $U\in \gamma_{k-1}$ and $x\in U$, there exists $V \in\gamma_{k}$ such that $x\oplus V\subseteq U$;
\item[(vii)] for every $U\in\gamma_{k}$ we have
  $\text{gyr}[x,y]U\subseteq U$ for each $x, y\in G$.
\end{enumerate}

Since $\gamma_n$ is countable, it's possible to identify a countable family $\lambda_{n,1}\subseteq\mathcal{U}^*(0_G)$ ensuring for every $U\in\gamma_n$, there exists $V\in\lambda_{n,1}$ for which $V \oplus V \subseteq U$.
Considering that the gyrogroup $G$ is $\omega$-balanced, it is possible to locate a countable family $\lambda_{n,2}\subseteq\mathcal{U}^*(0_G)$ that is subordinated to every $U\in\gamma_n$.
As $\gamma_n\subseteq\mathcal{U}^*(0_G)$, by the definition of $\omega$-good set, there exists a countable family $\lambda_{n,3}\subseteq\mathcal{U}^*(0_G)$ such that for any set $U\in \gamma_n$ and for every $x\in U$, there is $V\in \lambda_{n,3}$ with the property that $x \oplus V \subseteq U$.
Given that $G$ is a strongly paratopological gyrogroup, for each $U \in \mathcal{U}^*(0_G)$,
it holds that $\text{gyr}[x,y]U\subseteq U$ for all $x, y\in G$.
Define $\gamma_{n+1}$ to be the minimal family that containes $\gamma_n\cup\bigcup_{i=1}^3 \lambda_{n,i}$ and is also closed under finite intersections. Evidently, $\gamma_{n+1}$ is countable and fulfills the conditions (i) through (vii), thereby concluding our construction.

Observe that the set $\gamma = \bigcup_{n \in \omega} \gamma_n$ is countable and meets the requirements (a)-(c) as specified in Lemma \ref{lem3.17}.
Therefore, the set
$N=\bigcap \{U \cap (\ominus U): U \in \gamma\}$ forms an invariant subgyrogroup of $G$.
Consider the algebraic gyrogroup \(G/N \). By Theorem \ref{pro2.9s}, \(G/N\) is a gyrocommutative gyrogroup.
Define $p:G\to G/N$ as the canonical homomorphism.
And $p$ is continuous by Proposition \ref{pro2.28}.
Set $\mathcal{B}=\{p(V):V\in \gamma\}$, and let $H = G/N$.

We assert that the family $\mathcal{B}$ has the following characteristics:
\begin{enumerate}
\item[(1)] for each $A, B\in\mathcal{B}$ there exists $C\in\mathcal{B}$ such that $C\subseteq A\cap B$;
\item[(2)] for every $A \in\mathcal{B}$, there exists $B \in\mathcal{B}$ such that $B\oplus B\subseteq A$;
\item[(3)] for all $A \in\mathcal{B}$ and $a \in A$, there exists $B \in\mathcal{B}$ such that $a\oplus B\subseteq A$;
\item[(4)] $\mathcal{B}$ is subordinated to each $A \in\mathcal{B}$;
\item[(5)] for each $A\in\mathcal{B}$ and $a, b\in H$ we have
  $\text{gyr}[a,b]A\subseteq A$.
\end{enumerate}

Indeed, since $\gamma$ is closed under finite intersections, property (1) is satisfied. Statement (iv) supports property (2). Properties (3)-(5) are derived from statements (v)-(vii).
According to Theorem \ref{the}, the satisfaction of conditions (1)-(5) leads to the conclusion that it is possible to introduce a topology $\tau$ on $H$, whereby $(H, \tau)$ qualifies as a paratopological gyrogroup with $\mathcal{B}$ acting as a local base at $0_H$.
In fact, the neighborhood base
$\mathcal{B}$ at $0_H$ of $H$ witness that $H$ is a strongly paratopological gyrogroup.
The countability of $\mathcal{B}$ ensures that $(H, \tau)$ meets the conditions for being first-countable.

To show the paratopological gyrogroup $(H, \tau)$ is $T_0$, consider $y\in H\setminus\{0_H\}$, where $0_H$ is the
identity of $H$. Select $z\in G$ such that $p(z)=y$. The difference $y \neq 0_H$ implies that $z\notin N=\bigcap\{U\cap(\ominus U): U\in\gamma\}$.
Consequently, we can find $W\in\gamma$ for which $z\notin W\cap(\ominus W)$.
Choose $W_1\in\gamma$ that complies with $W_1\oplus W_1\subseteq W$. We assert that $y\notin p(W_1)\cap (\ominus p(W_1))$.
Otherwise it would be possible to identify elements $u, v\in W_1$ where $y=p(z)= p(u)=\ominus p(v)$, leading to $\ominus u\oplus z, v\oplus z\in N$. This situation would imply $z\in (u\oplus N) \cap (\ominus v\oplus N)\subseteq (W_1\oplus W_1)\cap(\ominus W_1\oplus (\ominus W_1))\subseteq W\cap(\ominus W)$, which is a contradiction of the selection of $W$. Thus, $(H,\tau)$ is $T_0$.

Finally, select $U\in\gamma$ such that $U\oplus U \subseteq U^*_0$.
The set $V_0 = p(U)$ then forms an open neighborhood of $0_H$ in $H$, with $p^{-1}(V_0)=U\oplus N\subseteq U\oplus U\subseteq U^*_0\subseteq U_0$.
\end{proof}

Utilizing the preceding theorems and the weakly Hausdorff number, we describe the subgyrogroups of products of $T_1$ paratopological gyrogroups that are first-countable.

\begin{theorem}\label{the3.30s}
A strongly paratopological gyrogroup
$G$ is topologically isomorphic to a subgyrogroup of a topological product
of first-countable $T_1$ strongly paratopological gyrogroups if and only if $G$ is $T_1$, $\omega$-balanced and $wHs(G)\leq\omega$.
\end{theorem}

\begin{proof}
Suppose that $G$ is a subgyrogroup of the product $\Pi=\prod_{i\in I}H_i$
of a family of first-countable $T_1$ strongly paratopological gyrogroups.
By Theorem \ref{the3.19} every first-countable paratopological gyrogroup $H_i$ is $\omega$-balanced.
Then the product gyrogroup $\Pi$ is $\omega$-balanced by Proposition \ref{pro3014s}.
Since by Proposition \ref{pro3.12s}, the property of being $\omega$-balanced is hereditary with respect to taking subgyrogroup, we conclude that $G$ is $\omega$-balanced as well.
By Propositions \ref{pro3.2s}, \ref{pro3.3s} and \ref{pro3.3} we
have that $wHs(G)\leq wHs(\Pi)\leq\omega$.

To prove the sufficiency, it suffices to verify that every $\omega$-balanced strongly paratopological gyrogroup
$G$ with $wHs(G)\leq\omega$ is projectively
$T_1$ first-countable.
Let the neighborhood base
$\mathcal{U}$ at $0_G$ of $G$ witness that $G$ is a strongly paratopological gyrogroup.
Consider an arbitrary $U_0\in\mathcal{U}$. We will construct a continuous homomorphism $p: G\rightarrow H$ onto a first-countable $T_1$ paratopological gyrogroup $H$
and find a neighbourhood $V_0$ of
the identity in $H$ such that $p^{-1}(V_0)\subseteq U_0$.

Let $\mathcal{U}^*(0_G)$ represent the subfamily of $\mathcal{U}$ that consists of $\omega$-good sets. According to Lemma \ref{lem3.16}, $\mathcal{U}^*(0_G)$ serves as a local base for $G$ at $0_G$. By induction, we aim to construct a sequence denoted as $\{\gamma_n: n\in\omega\}$, where each $\gamma_n\subseteq \mathcal{U}^*(0_G)$ for all $n\in\omega$.

Let $U_0 \in \mathcal{U}^*(0_G)$, and choose $U^*_0$ such that $U^*_0 \subseteq U_0$. Define $\gamma_0 = \{U^*_0\}$. Assume that for a given $n \in \omega$, the families $\gamma_0, \ldots, \gamma_n$ have been defined. These families are constructed to meet the following conditions for each index $k \leq n$:
\begin{enumerate}
\item[(i)] $\gamma_k \subseteq \mathcal{U}^*(0_G)$ and $|\gamma_k|\leq\omega$;
\item[(ii)] $\gamma_{k-1} \subseteq\gamma_k$;
\item[(iii)] $\gamma_k$ is closed under finite intersections;
\item[(iv)] for every $U\in \gamma_{k-1}$, there exists $V \in\gamma_{k}$ such that $V\oplus V\subseteq U$;
\item[(v)] the family $\gamma_k$ is subordinated to $U$, for each $U\in \gamma_{k-1}$;
\item[(vi)] for every $U\in \gamma_{k-1}$ and $x\in U$, there exists $V \in\gamma_{k}$ such that $x\oplus V\subseteq U$;
\item[(vii)] for every $U\in\gamma_{k}$ we have
  $\text{gyr}[x,y]U\subseteq U$ for each $x, y\in G$;
\item[(viii)] $\bigcap_{V\in\gamma_k}(\ominus V)\subseteq U$ for every $U\in\gamma_{k-1}$.
\end{enumerate}

Given that the family $\gamma_n$ is countable, we can select a countable subfamily $\lambda_{n,1} \subseteq \mathcal{U}^*(0_G)$ ensuring that for each set $U \in \gamma_n$ there exists $V \in \lambda_{n,1}$ satisfying $V \oplus V \subseteq U$.
Considering that the gyrogroup $G$ is $\omega$-balanced, it is possible to locate a countable family $\lambda_{n,2}\subseteq\mathcal{U}^*(0_G)$ that is subordinated to every $U\in\gamma_n$.
As $\gamma_n\subseteq\mathcal{U}^*(0_G)$, there exists a countable family $\lambda_{n,3}\subseteq\mathcal{U}^*(0_G)$ such that for every for any set $U\in \gamma_n$ and for every $x\in U$, there is $V\in \lambda_{n,3}$ with the property that $x \oplus V \subseteq U$.
Given that $G$ is a strongly paratopological gyrogroup, for each $U \in \mathcal{U}^*(0_G)$,
it holds that $\text{gyr}[x,y]U\subseteq U$ for all $x, y\in G$.
Based on the assumption $wHs(G)\leq\omega$, we can extract a countable subfamily $\lambda_{n,4} \subseteq\mathcal{U}^*(0_G)$ satisfying $\bigcap_{V\in\lambda_{n,4}}(\ominus V)\subseteq U$ for every $U\in\gamma_{n}$.
Define $\gamma_{n+1}$ to be the minimal family that containes $\gamma_n\cup\bigcup_{i=1}^4 \lambda_{n,i}$ and is also closed under finite intersections. Evidently, $\gamma_{n+1}$ is countable and fulfills the conditions (i)-(viii), thereby concluding our construction.

Observe that the set $\gamma = \bigcup_{n \in \omega} \gamma_n$ is countable and meets the requirements (a)-(d) as specified in Lemma \ref{lem3.17s}.
Therefore, the set
$N = \bigcap \{U: U \in \gamma\}$ forms an invariant subgyrogroup of $G$.
Let's consider the algebraic gyrogroup $G/N$.
Define $p: G \to G/N$ as the canonical homomorphism. Set $\mathcal{B} = \{p(V) : V \in \gamma\}$, and let $H = G/N$.

We assert that the family $\mathcal{B}$ has the following characteristics:
\begin{enumerate}
\item[(1)] for each $A, B\in\mathcal{B}$ there exists $C\in\mathcal{B}$ such that $C\subseteq A\cap B$;
\item[(2)] for every $A \in\mathcal{B}$, there exists $B \in\mathcal{B}$ such that $B\oplus B\subseteq A$;
\item[(3)] for all $A \in\mathcal{B}$ and $a \in A$, there exists $B \in\mathcal{B}$ such that $a\oplus B\subseteq A$;
\item[(4)] $\mathcal{B}$ is subordinated to each $A \in\mathcal{B}$;
\item[(5)] for each $A\in\mathcal{B}$ and $a, b\in H$ we have
  $\text{gyr}[a,b]A\subseteq A$.
\end{enumerate}

Indeed, since $\gamma$ is closed under finite intersections, property (1) is satisfied. Statement (iv) supports property (2). Properties (3)-(5) are derived from statements (v)-(vii).
According to Theorem \ref{the}, the satisfaction of conditions (1)-(5) leads to the conclusion that it is possible to introduce a topology $\tau$ on $H$, whereby $(H, \tau)$ qualifies as a paratopological gyrogroup with $\mathcal{B}$ acting as a local base at $0_H$.
In fact, the neighborhood base
$\mathcal{B}$ at $0_H$ of $H$ witness that $H$ is a strongly paratopological gyrogroup.
The countability of $\mathcal{B}$ ensures that $(H, \tau)$ meets the conditions for being first-countable.

To demonstrate that the paratopological gyrogroup $(H,\tau)$ is $T_1$, consider $y\in H\setminus\{0_H\}$, where $0_H$ denotes the identity of $H$. Select $x\in G$ such that $p(x)= y$. Given that $y \neq 0_H$, it follows that $x\notin N=\bigcap\{V: V\in\gamma\}$.
Consequently,
there exists $V\in\gamma$ such that $x\notin V$.
Choose $W\in\gamma$ that meets the condition $W \oplus W \subseteq V$. We assert that $y\notin p(W)$.
Alternatively, suppose we can find $u\in W$ such that $y = p(x) = p(u)$. Consequently, $\ominus u\oplus x\in N$. This implies that $x\in u\oplus N\subseteq W\oplus W\subseteq V$, which contradicts the selection of $V$. Therefore, we can conclude that $(H,\tau)$ is $T_1$.

Finally, select $U\in\gamma$ such that $U\oplus U \subseteq U^*_0$.
The set $V_0 = p(U)$ then forms an open neighborhood of $0_H$ in $H$, with $p^{-1}(V_0)=U\oplus N\subseteq U\oplus U\subseteq U^*_0\subseteq U_0$.
\end{proof}

\begin{corollary}
Every $\omega$-balanced, $T_1$ strongly paratopological gyrogroup $G$, which possesses a countable pseudocharacter, supports a continuous isomorphism mapping onto a first-countable $T_1$ strongly paratopological gyrogroup.
\end{corollary}
\begin{proof}
Given that $\psi(G) \leq \omega$, it follows that $wHs(G) \leq \omega$. According to Theorem \ref{the3.30}, $G$ is topologically isomorphic to a subgyrogroup of a topological product $H =\prod_{i \in I} H_i$, where each $H_i$ is a first-countable, $T_1$, strongly paratopological gyrogroup. Owing to $G$'s countable pseudocharacter, a countable subset $J \subseteq I$ can be identified such that the projection $p_J: H \rightarrow \prod_{i \in J} H_i$ restricted to $G$ serves as a continuous monomorphism.
Hence, $p_J|_G: G \rightarrow p_J(G)$ constitutes the necessary continuous isomorphism.
\end{proof}

\begin{theorem}\label{the3.30}
A strongly paratopological gyrogroup
$G$ is topologically isomorphic to a subgyrogroup of a topological product
of first-countable Hausdorff strongly paratopological gyrogroups if and only if $G$ is Hausdorff, $\omega$-balanced and $Hs(G)\leq\omega$.
\end{theorem}

\begin{proof}
Suppose that $G$ is a subgyrogroup of the product $\Pi=\prod_{i\in I}H_i$
of a family of first-countable paratopological gyrogroups.
By Theorem \ref{the3.19} every first-countable paratopological gyrogroup $H_i$ is $\omega$-balanced.
Then the product gyrogroup $\Pi$ is $\omega$-balanced by Proposition \ref{pro3014s}.
Since by Proposition \ref{pro3.12s}, the property of being $\omega$-balanced is hereditary with respect to taking subgyrogroup, we conclude that $G$ is $\omega$-balanced as well.
By Propositions \ref{pro3.2s1},
\ref{pro3.3s1} and \ref{pro3.4s1} we
have that $Hs(G)\leq Hs(\Pi)\leq\omega$.

To prove the sufficiency, it suffices to verify that every $\omega$-balanced strongly paratopological gyrogroup
$G$ with $Hs(G)\leq\omega$ is projectively
Hausdorff first-countable.
Let the neighborhood base
$\mathcal{U}$ at $0_G$ of $G$ witness that $G$ is a strongly paratopological gyrogroup.
Consider an arbitrary $U_0\in\mathcal{U}$. We will construct a continuous homomorphism $p: G\rightarrow H$ onto a first-countable Hausdorff paratopological gyrogroup $H$
and find a neighbourhood $V_0$ of
the identity in $H$ such that $p^{-1}(V_0)\subseteq U_0$.

Let $\mathcal{U}^*(0_G)$ represent the subfamily of $\mathcal{U}$ that consists of $\omega$-good sets. According to Lemma \ref{lem3.16}, $\mathcal{U}^*(0_G)$ serves as a local base for $G$ at $0_G$. By induction, we aim to construct a sequence denoted as $\{\gamma_n: n\in\omega\}$, where each $\gamma_n\subseteq \mathcal{U}^*(0_G)$ for all $n\in\omega$.

Let $U_0 \in \mathcal{U}^*(0_G)$, and choose $U^*_0$ such that $U^*_0 \subseteq U_0$. Define $\gamma_0 = \{U^*_0\}$. Assume that for a given $n \in \omega$, the families $\gamma_0, \ldots, \gamma_n$ have been defined. These families are constructed to meet the following conditions for each index $k \leq n$:
\begin{enumerate}
\item[(i)] $\gamma_k \subseteq \mathcal{U}^*(0_G)$ and $|\gamma_k|\leq\omega$;
\item[(ii)] $\gamma_{k-1} \subseteq\gamma_k$;
\item[(iii)] $\gamma_k$ is closed under finite intersections;
\item[(iv)] for every $U\in \gamma_{k-1}$, there exists $V \in\gamma_{k}$ such that $V\oplus V\subseteq U$;
\item[(v)] the family $\gamma_k$ is subordinated to $U$, for each $U\in \gamma_{k-1}$;
\item[(vi)] for every $U\in \gamma_{k-1}$ and $x\in U$, there exists $V \in\gamma_{k}$ such that $x\oplus V\subseteq U$;
\item[(vii)] for every $U\in\gamma_{k}$ we have
  $\text{gyr}[x,y]U\subseteq U$ for each $x, y\in G$;
\item[(viii)] $\bigcap_{V\in\gamma_k}(V\boxminus V)\subseteq U$ for every $U\in\gamma_{k-1}$.
\end{enumerate}

Given that the family $\gamma_n$ is countable, we can select a countable subfamily $\lambda_{n,1} \subseteq \mathcal{U}^*(0_G)$ ensuring that for each set $U \in \gamma_n$ there exists $V \in \lambda_{n,1}$ satisfying $V \oplus V \subseteq U$.
Considering that the gyrogroup $G$ is $\omega$-balanced, it is possible to locate a countable family $\lambda_{n,2}\subseteq\mathcal{U}^*(0_G)$ that is subordinated to every $U\in\gamma_n$.
As $\gamma_n\subseteq\mathcal{U}^*(0_G)$, there exists a countable family $\lambda_{n,3}\subseteq\mathcal{U}^*(0_G)$ such that for every for any set $U\in \gamma_n$ and for every $x\in U$, there is $V\in \lambda_{n,3}$ with the property that $x \oplus V \subseteq U$.
Given that $G$ is a strongly paratopological gyrogroup, for each $U \in \mathcal{U}^*(0_G)$,
it holds that $\text{gyr}[x,y]U\subseteq U$ for all $x, y\in G$.
Based on the assumption $Hs(G)\leq\omega$, we can extract a countable subfamily $\lambda_{n,4} \subseteq\mathcal{U}^*(0_G)$ satisfying $\bigcap_{V\in\lambda_{n,4}}(V\boxminus V)\subseteq U$ for every $U\in\gamma_{n}$.
Define $\gamma_{n+1}$ to be the minimal family that containes $\gamma_n\cup\bigcup_{i=1}^4 \lambda_{n,i}$ and is also closed under finite intersections. Evidently, $\gamma_{n+1}$ is countable and fulfills the conditions (i) through (viii), thereby concluding our construction.

Since $Hs(G)\leq\omega$, we can get $wHs(G)\leq Hs(G)\leq\omega$ by Proposition \ref{pro3.2s}.
Thus the set $\gamma = \bigcup_{n \in \omega} \gamma_n$ is countable and meets the requirements (a)-(d) as specified in Lemma \ref{lem3.17s}.
Therefore, the set
$N = \bigcap \{U: U \in \gamma\}$ forms an invariant subgyrogroup of $G$.
Let's consider the algebraic gyrogroup $G/N$.
Define $p: G \to G/N$ as the canonical homomorphism. Set $\mathcal{B} = \{p(V) : V \in \gamma\}$, and let $H = G/N$.

We assert that the family $\mathcal{B}$ has the following characteristics:
\begin{enumerate}
\item[(1)] for each $A, B\in\mathcal{B}$ there exists $C\in\mathcal{B}$ such that $C\subseteq A\cap B$;
\item[(2)] for every $A \in\mathcal{B}$, there exists $B \in\mathcal{B}$ such that $B\oplus B\subseteq A$;
\item[(3)] for all $A \in\mathcal{B}$ and $a \in A$, there exists $B \in\mathcal{B}$ such that $a\oplus B\subseteq A$;
\item[(4)] $\mathcal{B}$ is subordinated to each $A \in\mathcal{B}$;
\item[(5)] for each $A\in\mathcal{B}$ and $a, b\in H$ we have
  $\text{gyr}[a,b]A\subseteq A$.
\end{enumerate}

Indeed, since $\gamma$ is closed under finite intersections, property (1) is satisfied. Statement (iv) supports property (2). Properties (3)-(5) are derived from statements (v)-(vii).
According to Theorem \ref{the}, the satisfaction of conditions (1)-(5) leads to the conclusion that it is possible to introduce a topology $\tau$ on $H$, whereby $(H, \tau)$ qualifies as a paratopological gyrogroup with $\mathcal{B}$ acting as a local base at $0_H$.
In fact, the neighborhood base
$\mathcal{B}$ at $0_H$ of $H$ witness that $H$ is a strongly paratopological gyrogroup.
The countability of $\mathcal{B}$ ensures that $(H, \tau)$ meets the conditions for being first-countable.

To demonstrate that the paratopological gyrogroup $(H,\tau)$ is Hausdorff, consider $y\in H\setminus\{0_H\}$, where $0_H$ denotes the identity of $H$. Select $x\in G$ such that $p(x)= y$. Given that $y \neq 0_H$, it follows that $x\notin N=\bigcap_{V\in\gamma}V\supseteq \bigcap_{V\in\gamma}(V\boxminus V)$.
Consequently,
there exists $V\in\gamma$ such that $x\notin V\boxminus V$, that is, $V \cap(x\oplus V)=\emptyset$.
Choose $W\in\gamma$ that meets the condition $W \oplus W \subseteq V$. We assert that $p(W) \cap(y\oplus p(W))=\emptyset$.
Alternatively, suppose we can find $u, v\in W$ such that $p(u)=y\oplus p(v)$. Consequently, $\ominus u\oplus (x\oplus v)\in N$. This implies that $x\in (u\oplus N)\boxminus v\subseteq(u\oplus W)\boxminus v
\subseteq(W\oplus W)\boxminus W\subseteq V\boxminus V$, which contradicts the selection of $V$. Therefore, we can conclude that $(H,\tau)$ is Hausdorff.

Finally, select $U\in\gamma$ such that $U\oplus U \subseteq U^*_0$.
The set $V_0 = p(U)$ then forms an open neighborhood of $0_H$ in $H$, with $p^{-1}(V_0)=U\oplus N\subseteq U\oplus U\subseteq U^*_0\subseteq U_0$.
\end{proof}

In our Theorems \ref{the3.26}, \ref{the3.30s}, and \ref{the3.30}, we comprehensively detail the projectively first-countable $T_i$ paratopological gyrogroups for $i = 0, 1$. This poses the following questions:
\begin{question}
Provide an intrinsic description of projectively second-countable $T_i$ paratopological gyrogroups for $i= 0, 1, 2, 3, 3\frac{1}{2}$.
\end{question}

\section{Other related results}
Drawing a complete parallel with the concept in group theory, the authors of \cite{ST} investigate substructures within a gyrogroup that are formed by a subset. Specifically, their focus is on the cyclic subgyrogroups that arise from the generation by a single element. The subgyrogroup generated by one-element set $\{a\}$ is called the {\it cyclic
subgyrogroup generated by} $a$, which will be denoted by $\langle a\rangle$.
The explicit description of $\langle a\rangle$ is as following:

Let $G$ be a gyrogroup and let $a$
be an element of $G$. For $m\in \mathbb{Z}$, define recursively the following notation:
$$0\cdot a=0, m\cdot a=a\oplus((m-1)\cdot a), m\geq1, m\cdot a=(-m)\cdot (\ominus a), m<0,$$
$$a\cdot 0=0, a\cdot m=(a\cdot (m-1))a\oplus a, m\geq1, a\cdot m=(\ominus a)\cdot (-m), m<0,$$
and $a\cdot m= m\cdot a$ for all
$m\in \mathbb{Z}$.

\begin{example}
In \cite{ST} T. Suksumran exhibited the gyrogroup $G_{8}=\{0,1,\ldots,7\}$, whose addition table is presented in Table 1.
In $G_{8}$, there are two gyroautomorphisms denoted by $A$ and $I$.
The transformation of $A$ is given in cyclic notation by $A=(4,6)(5,7)$ and $I$ stands for
the identity transformation.
The gyration table for $G_{8}$ is presented in Table 2.
It is easily to see that $\langle 1\rangle=\{0,1,2,3\}$ is the cyclic
subgroup of $G_{8}$.
From the calculations, it can be determined that $2\cdot 1=1\oplus1=3$, $3\cdot 1=1\oplus(1\oplus1)=2$, $4\cdot 1=1\oplus(1\oplus(1\oplus1))=0$.
\end{example}

\begin{center}
\begin{tabular}{|c|cccccccc|}
\hline
$\oplus$&0& 1  & 2 & 3 & 4  & 5 & 6 & 7\\
\hline
0&0 & 1  & 2 & 3 & 4  & 5 & 6 & 7 \\
1 &  1& 3& 0& 2& 7& 4 &5 &6\\
2 & 2 &0 &3& 1& 5& 6& 7& 4\\
3 & 3 &2 &1 &0& 6& 7& 4 &5\\
4 &4 &5 &7 &6& 3& 2& 0 &1\\
5&5& 6& 4 &7& 2& 0& 1& 3\\
6&6& 7 &5 &4 &0& 1 &3& 2\\
7&7& 4& 6 &5& 1& 3& 2 &0\\
\hline
\end{tabular}
\end{center}
\begin{center}
 {Table 1. Addition table for the gyrogroup $G_{8}$ \cite{ST}.}
\end{center}

\begin{center}
\begin{tabular}{|c|cccccccc|}
\hline
gyr &0& 1  & 2 & 3 & 4  & 5 & 6 & 7\\
\hline
0&$I$ & $I$  &$I$ &$I$& $I$& $I$ &$I$&$I$\\
1 & $I$&$I$& $I$& $I$&$A$& $A$& $A$& $A$ \\
2 & $I$&$I$& $I$& $I$&$A$& $A$& $A$& $A$ \\
3 & $I$ &$I$ &$I$& $I$& $I$& $I$& $I$& $I$\\
4 &$I$ &$A$& $A$&$I$& $I$& $A$& $I$& $A$\\
5 &$I$ &$A$& $A$&$I$&  $A$& $I$& $A$& $I$\\
6 &$I$ &$A$& $A$&$I$&  $I$& $A$& $I$& $A$\\
7 &$I$ &$A$& $A$&$I$&  $A$& $I$& $A$& $I$\\
\hline
\end{tabular}
\end{center}
\begin{center}
 {Table 2. Gyration table for $G_{8}$ \cite{ST}.}
\end{center}

The authors have proved the following theorems, which will be instrumental in the subsequent proof process.
\begin{theorem}\cite{ST}\label{the4.10}
Let $a$ be an element of a gyrogroup. For all $m, k\in \mathbb{Z}$, $(m\cdot a)\oplus (k\cdot a)=(m+k)\cdot a$.
\end{theorem}
\begin{theorem}\cite{ST}\label{the4.10s}
 Any gyrogroup generated by one element is a cyclic group.
\end{theorem}

\begin{definition}\cite{JX2}
A paratopological gyrogroup $G$ is called topologically periodic if for each
$x\in G$ and every neighborhood $U$ of the identity there exists an integer $n$ such that
$n\cdot x\in U$.
\end{definition}
It is evident that the cyclic subgyrogroup generated by $a$
within a paratopological gyrogroup is topologically periodic.

\begin{proposition}\label{pro3.4}
Every topologically periodic paratopological gyrogroup $G$ satisfies the inequality $wHs(G)\leq\omega$.
\end{proposition}
\begin{proof}
Consider $U$ as a given neighborhood around the identity element 0 in the gyrogroup $G$, and let $\{V_i : i \in \omega\}$ represent a sequence of neighborhoods of 0 where $V_0 = U$ and $V_{i+1}\oplus V_{i+1}\subseteq V_i$ for every $i\in\omega$.

We aim to show that the $F=\bigcap\{\ominus V_i:i\in\omega\}\subseteq U$. This result would indicate that $wHs(G)\leq\omega$. Suppose $x\in F$, it follows that $x\in \ominus V_i$ for every $i\in \omega$. Consequently, $\ominus x\in V_i$ for all $i\in \omega$.
Since $G$ is topologically periodic, it is possible to select $n\in \omega\setminus\{0\}$ so that $n\cdot x\in V_1$.
If $n =1$, then $x\in U$ is obvious. Thus we can assume that $n>1$. Clearly, $(n-1)\cdot(\ominus x)\in \underbrace{((V_{i}\oplus V_{i})\oplus V_{i})\cdots \oplus V_{i}}_{n-1}$ for each $i\in\omega$.
From $V_{i+1}\oplus V_{i+1}\subseteq V_i$ for every $i\in\omega$, it follows that
$(((V_{n-1}\oplus V_{n-1})\oplus V_{n-2})\oplus\ldots\oplus V_3)\oplus V_2\subset V_1$.
Select an element $i_0\in \omega$ such that $i_0\geq n-1$. With this choice, it follows that $\underbrace{((V_{i_0}\oplus V_{i_0})\oplus V_{i_0})\cdots \oplus V_{i_0}}_{n-1}\subseteq V_1$.
By Theorem \ref{the4.10} we can get $(n\cdot x)\oplus((n-1)\cdot(\ominus x))=(n\cdot x)\oplus(-(n-1)\cdot x)=x$.
Then $x=(n\cdot x)\oplus((n-1)\cdot(\ominus x))\in V_1\oplus (\underbrace{((V_{i_0}\oplus V_{i_0})\oplus V_{i_0})\cdots \oplus V_{i_0}}_{n-1})\subseteq V_1\oplus V_1\subseteq U$.
\end{proof}

\begin{corollary}
The inequality $wHs(H)\leq\omega$ is satisfied for every cyclic subgyrogroup $H$ of the gyrogroup $G$.
\end{corollary}

Let $\tau$ be an infinitely cardinal number. Recall that a space $X$ is a $P_{\tau}$-{\it space} if for any family $\gamma$ of open neighborhoods such that $|\gamma|\leq\tau$ the set $\bigcap_{U\in \gamma} U$ is open in $X$.
If $\tau=\omega$, then $X$ is called a {\it P-space}.

\begin{theorem}\label{the3.9}
Consider $G$ as a strongly paratopological gyrogroup, and let $\tau=wHs(G)$.
If $G$ is also a $P_{\tau}$-space, then $G$ is a strongly topological gyrogroup.
\end{theorem}
\begin{proof}
In order to demonstrate that $G$ qualifies as a strongly topological gyrogroup, it suffices to verify the continuity of the inverse operation at the identity $0$ in $G$, as established by Lemma \ref{pro3.1}.

Consider any arbitrary open neighborhood $U$ of the identity element 0 in $G$.
Given that $\tau=wHs(G)$, there exists a family $\gamma$ of open neighborhoods at 0 such that $|\gamma|\leq\tau$ and $\bigcap_{V\in \gamma} (\ominus V)\subseteq U$.
This implies that $\ominus\bigcap_{V\in\gamma}V=\bigcap_{V\in\gamma}(\ominus V)\subseteq U$.
Since $G$ is a $P_{\tau}$-space, it follows that $\bigcap_{V\in\gamma}V$ is an open set in $G$. Hence $G$ is a strongly topological gyrogroup.
\end{proof}

\begin{corollary}\label{the3.10}
Every strongly paratopological gyrogroup $G$ that is topologically periodic and qualifies as a
$P$-space is also a strongly topological gyrogroup.
\end{corollary}
\begin{proof}
This assertion can be directly inferred from Proposition \ref{pro3.4} and Theorem \ref{the3.9}.
\end{proof}

\begin{lemma}\cite{JX2}\label{lem3.11}
Let the neighborhood base
$\mathcal{U}$ at 0 of $G$ witness that $G$ is a strongly paratopological gyrogroup.
If $V\oplus V\subseteq U$ where $U, V\in \mathcal{U}$,
then $\ominus(\overline{\ominus V})\subseteq U$.
\end{lemma}

\begin{corollary}\label{the3.10}
Assume $G$ is a $T_1$ strongly paratopological gyrogroup and also a $P$-space that satisfies one of the following conditions.
Then $G$ is a strongly topological gyrogroup.
\begin{enumerate}
\item[(1)] $G$ is Lindel\"{o}f;
\item[(2)] there erists a nonempty open set $V$ in $G$ such that $\ominus V$ is Lindel\"{o}f.
\end{enumerate}
\end{corollary}
\begin{proof}
Based on Proposition \ref{pro3.4s} and Theorem \ref{the3.9}, it suffices to demonstrate the existence of an open neighborhood $W$ of 0 in $G$ for which $\overline{\ominus W}$ is a Lindel\"{o}f space.
For the homogeneity of $G$, we can also postulate that the open set $V$ includes the identity element 0.
Consequently, to satisfy the requirements, one may simply select an open neighborhood $W$ of 0 with the property that $W\oplus W\subseteq V$. From this, by Lemma \ref{lem3.11}, we infer that $\overline{\ominus W}\subseteq \ominus V$.
Since $\ominus V$ is Lindel$\ddot{o}$f, so is $\overline{\ominus W}$.
\end{proof}

Let $\mathcal{P}$ be a topological property and $G$ a paratopological gyrogroup.
We say that $G$ is {\it co-local} $\mathcal{P}$ (i.e., {\it conjugate-local} $\mathcal{P}$) if there exists a nonempty open set $V$ in $G$ such that $\overline{\ominus V}$ has the property $\mathcal{P}$.

\begin{theorem}\label{the3.13}
If $G$ is a strongly paratopological gyrogroup that is co-locally countably compact with $wHs(G)\leq\omega$, then $G$ is a strongly topological gyrogroup.
\end{theorem}
\begin{proof}
Given that $G$ is a paratopological gyrogroup, the homogeneity of $G$ allows us to postulate the existence of an open neighborhood $V$ of the identity 0 in $G$, for which $\overline{\ominus V}$ is countable compact.

Consider an open neighborhood $U$ of the identity 0. Given $wHs(G)\leq\omega$, it is possible to identify a family $\{W_i:i\in\omega\}$ of open neighborhoods of 0, with the property $\bigcap_{i\in\omega}(\ominus W_i)\subseteq U$.
Additionally, it can be presumed that the family
$\{W_i:i\in\omega\}$ satisfies $W_{i+1}\oplus W_{i+1}\subseteq W_i$ for each $i\in\omega$.
Referring to Lemma \ref{lem3.11}, it can be established that $\bigcap_{i\in\omega}\overline{\ominus W_i}=\bigcap_{i\in\omega}(\ominus W_i)\subseteq U$.
Since $\overline{\ominus V}\setminus U$ is countably compact, we can find an index $i_0\in\omega$ such that $\ominus W_{i_0}\cap(\overline{\ominus V}\setminus U)=\emptyset$.
Define $W=W_{i_0}\cap V$. It is then readily apparent that $W$ includes 0 and it follows that $\ominus W=\ominus W_{i_0}\cap (\ominus V)\subseteq U$. This inclusion implies that $G$ is a topological gyrogroup according to Lemma \ref{pro3.1}.
\end{proof}

\begin{corollary}\label{the3.14}
Suppose $G$ is a strongly paratopological gyrogroup that satisfies one of the following conditions. In that case, $G$ is a strongly topological gyrogroup.
\begin{enumerate}
\item[(1)] $G$ is countably compact with $wHs(G) \leq\omega$;
\item[(2)] $G$ is a $T_1$-space with the existence of a nonempty open set $V$ such that $\ominus V$ is countably compact and $\psi(G) \leq \omega$;
\item[(3)] $G$ is Hausdorff countably compact and topologically periodic \cite[Theorem 4.12]{JX2};
\item[(4)] $G$ is topologically periodic with an existing nonempty open set $V$ where $\ominus V$ is countably compact;
\item[(5)] $G$ is Hausdorff compact \cite[Theorem 3.1]{JX2};
\item[(6)] $G$ is a Hausdorff space and possesses a nonempty open set $V$ such that $\ominus V$ is compact.
\end{enumerate}
\end{corollary}
\begin{proof}
It is readily apparent that that $(3)\Rightarrow(4)$ and $(5)\Rightarrow(6)$.

(1)The statement is a direct corollary of Theorem \ref{the3.13}.

(2)The result is a consequence of Proposition \ref{pro3.3}, Lemma \ref{lem3.11}, and Theorem \ref{the3.13}.

(4)The statement is a direct consequence of Proposition \ref{pro3.4}, Lemma \ref{lem3.11}, and Theorem \ref{the3.13}.

(6)The statement can be directly deduced from Proposition \ref{pro3.4s}, Lemma \ref{lem3.11}, and Theorem \ref{the3.13}.
\end{proof}

A sequence $\{U_n: n \in\omega\}$ of subsets of a space $X$ is decreasing if $U_n\supset U_{n+1}$ for each $n\in\omega$.
A semitopological gyrogroup $G$ is 2-{\it pseudocompact} if $\bigcap_{n\in\omega}\overline{\ominus U_n}\neq\emptyset$ for any
decreasing sequence $\{U_n: n \in\omega\}$ of nonempty open subsets of $G$.
Clearly, every countably compact paratopological gyrogroup is 2-pseudocompact. Hence, Theorem \ref{the3.15}
provides an enhancement to (1) of Corollary \ref{the3.14}.

\begin{theorem}\label{the3.15}
If $G$ is a 2-pseudocompact strongly paratopological gyrogroup, then $G$ is a strongly topological gyrogroup if and only if $wHs(G)\leq\omega$.
\end{theorem}
\begin{proof}
If $G$ is a topological gyrogroup, it follows that $wHs(G)=1$.

To establish the sufficiency part, consider an open neighborhood base $\mathcal{U}$ at the identity element 0 in $G$, demonstrating that $G$ satisfies the properties of a strongly paratopological gyrogroup.
Take $U\in \mathcal{U}$. Since
the operator $\oplus$ in $G$ is continuous, there exists $V \in \mathcal{U}$ such that $V\oplus V\subseteq U$.
Assuming the hypothesis $wHs(G)\leq\omega$, we can select a sequence \(\{V_n : n \in \omega\} \subseteq \mathcal{U}\) such that $\bigcap_{n\in\omega}(\ominus V_n)\subseteq V$. It can be assumed that for each $n\in\omega$, the relation $V_{n+1}\oplus V_{n+1}\subseteq V_{n}$ holds.
It follows that $\overline{\ominus V_{n+1}} \subseteq\ominus V_n$ for every $n\in\omega$ by Lemma \ref{lem3.11}
and, hence, $\bigcap_{n\in\omega}\overline{\ominus V_{n}}=\bigcap_{n\in\omega}(\ominus V_{n})$.
If $V_m \subseteq \overline{\ominus V}$ for some $m \in \omega$, it follows from Lemma \ref{lem19} that $V_m \subseteq \overline{\ominus V} \subseteq \ominus V\ominus V \subseteq \ominus U$.
Hence, we conclude that $\ominus V_m\subseteq U$, completing this part of the argument. Now, consider the case where $V_n\setminus \overline{\ominus V}\neq\emptyset$ for each $n\in\omega$.
Then we have
$$\bigcap_{n\in\omega}\overline{\ominus V_n\setminus(\ominus(\overline{\ominus V}))}\subseteq \bigcap_{n\in\omega}\overline{\ominus V_n\setminus V}\subseteq(\bigcap_{n\in\omega}\overline{\ominus V_n})\setminus V
=(\bigcap_{n\in\omega}(\ominus V_n))\setminus V=\emptyset.$$
This contradicts the 2-pseudocompactness of $G$, leading to the conclusion that $G$ is a strongly topological gyrogroup.
\end{proof}

\begin{corollary}
Let $G=\prod_{i\in I}G_i$ be the product of Lindel\"{o}f $T_1$ strongly paratopological gyrogroups. If $K$ is a 2-pseudocompact subgyrogroup of $G$, then $K$ is a strongly topological gyrogroup.
\end{corollary}

\begin{proof}
Propositions \ref{pro3.2s}, \ref{pro3.3s}, and Corollary \ref{cor3.7} establish that $wHs(K)\leq\omega$. Given that $K$ is 2-pseudocompact, the application of Theorem \ref{the3.15} confirms that $K$ is a strongly topological gyrogroup.
\end{proof}

\begin{corollary}
Each 2-pseudocompact strongly paratopological gyrogroup $G$ with $\psi(G)\leq\omega$ is a strongly topological gyrogroup.
\end{corollary}

\begin{proof}
This directly follows from Proposition \ref{pro3.3} and Theorem \ref{the3.15}.
\end{proof}

\section{Acknowledgments}
We are immensely thankful to the reviewers for their extensive feedback and suggestions on our paper, and for their dedicated efforts to improve its overall quality.
\vskip0.9cm

\end{document}